\documentclass[final,leqno,onefignum,onetabnum,nohypdvips]{siamart0516}
\pdfoutput=1

\usepackage{algpseudocode}
\usepackage{amsfonts}
\usepackage{amsmath}
\usepackage{booktabs}
\usepackage{caption}
\usepackage{graphicx}
\usepackage{mathtools}
\usepackage[referable]{threeparttablex}
\usepackage{tikz}
\usetikzlibrary{trees}

\crefname{appendix}{}{}

\usepackage{amssymb,amsfonts,amsmath}

\usepackage[mathscr]{eucal}

\usepackage{stmaryrd}

\usepackage{amsbsy}

\usepackage{bm}

\usepackage{braket}

\newcommand{\Dsquare}[1]{\left\llbracket #1 \right\rrbracket}

\newcommand{\V}[2][]{{\bm{#1\mathbf{\MakeLowercase{#2}}}}} 
\newcommand{\VE}[3][]{#1{\MakeLowercase{#2}}_{#3}} 
 
\newcommand{\VnE}[4][]{#1{\MakeLowercase{#2}}^{(#3)}_{#4}} 

\newcommand{\M}[2][]{{\bm{#1\mathbf{\MakeUppercase{#2}}}}} 
\newcommand{\Mn}[3][]{{\bm{#1\mathbf{\MakeUppercase{#2}}}}^{(#3)}} 
 
\newcommand{\ME}[3][]{#1{\MakeLowercase{#2}}_{#3}} 
\newcommand{\MnE}[4][]{#1{\MakeLowercase{#2}}^{(#3)}_{#4}} 

\newsiamthm{example}{Example}
\let\olddefinition\example
\renewcommand{\example}{\olddefinition\normalfont}

\newcommand{\eg}{e.g.}
\newcommand{\ie}{i.e.}
\newcommand{\etal}{et al.}

\newcommand{\diag}[1]{\operatorname{diag}\parenths*{#1}}
\newcommand{\bc}[3]{\lt\langle#1, #2, #3\rt\rangle}
\newcommand{\alg}[3]{\Dsquare{#1, #2, #3}}

\newcommand{\VP}[2]{\M{\cal P}_{#1,#2}}

\newcommand{\qmax}{Q}
\newcommand{\emax}{E}

\newcommand{\II}{\mathbb{I}}
\newcommand{\lt}{\left}
\newcommand{\rt}{\right}

\newcommand{\leftscalemat}{\M{D}_{A}}
\newcommand{\rightscalemat}{\M{D}_{B}}
\newcommand{\inscalemat}{\M{D}}
\DeclarePairedDelimiter{\abs}{\lvert}{\rvert}
\DeclarePairedDelimiter{\norm}{\lVert}{\rVert}
\DeclarePairedDelimiter{\parenths}{\lparen}{\rparen}

\newcommand{\Ostep}{\emph{O}~step}
\newcommand{\Osteps}{\emph{O}~steps}
\newcommand{\Istep}{\emph{I}~step}
\newcommand{\Isteps}{\emph{I}~steps}
\newcommand{\Sn}[2]{#1^{(#2)}} 

\newcommand{\pptk}[2]{p^{\prime(#1)}_{#2}}
\newcommand{\rpti}[2]{r^{\prime(#1)}_{#2}}
\newcommand{\sptj}[2]{s^{\prime(#1)}_{#2}}

\newcommand{\At}[1][t]{\Mn{A\!}{#1}}
\newcommand{\Bt}[1][t]{\Mn{B}{#1}}

\newcommand{\partdimm}{M_{0}}
\newcommand{\partdimk}{K_{0}}
\newcommand{\partdimn}{N_{0}}

\newcommand{\smallspace}{\vspace{0.2cm}}
\newcommand{\numbereq}{\stepcounter{equation}\tag{\theequation}}

\newcommand{\smallentry}{z}

\graphicspath{{./FIG/}}

\title{Improving the numerical stability of\\fast matrix multiplication}

\author{
Grey Ballard\thanks{Sandia National Laboratories, Livermore, California.  Current address: Computer Science Department, Wake Forest University, Winston Salem, North Carolina  (\email{ballard@wfu.edu})}.
This author's work was supported by an appointment to the Sandia National Laboratories Truman Fellowship in National Security Science and Engineering, sponsored by Sandia Corporation (a wholly owned subsidiary of Lockheed Martin Corporation) as Operator of Sandia National Laboratories under its U.S. Department of Energy Contract No. DE-AC04-94AL85000.
\and
Austin R.~Benson\thanks{Institute for Computational and Mathematical Engineering, Stanford University, Stanford, California. (\email{arbenson@stanford.edu})}.
This author's work
was supported by an Office of Technology Licensing Stanford Graduate Fellowship.
\and
Alex Druinsky\thanks{Lawrence Berkeley National Laboratory, Berkeley, California. (\email{adruinsky@lbl.gov})}
This material is based upon work supported by the US Department of Energy (DOE), Office of Science, Office of Advanced Scientific Computing Research (ASCR), Applied Mathematics program under contract number DE-AC02-05CH11231.
\and
Benjamin Lipshitz\thanks{San Bruno, California (\email{benjamin.lipshitz@gmail.com}).}
\and
Oded Schwartz\thanks{The Hebrew University, Jerusalem, Israel (\email{odedsc@cs.huji.ac.il})
Research is supported by grants 1878/14 and 1901/14 from the Israel Science Foundation (founded by the Israel Academy of Sciences and Humanities) and grant 3-10891 from the Ministry of Science and Technology, Israel; the Einstein Foundation and the Minerva Foundation; the Intel Collaborative Research Institute for Computational Intelligence (ICRI-CI); a grant from the United States-Israel Binational Science Foundation (BSF), Jerusalem, Israel; and the HUJI Cyber Security Research Center in conjunction with the Israel National Cyber Bureau in the Prime Minister's Office.
}
}

\begin{document}
\maketitle
\slugger{simax}{xxxx}{xx}{x}{x--x}

\begin{abstract}
Fast algorithms for matrix multiplication, namely those that perform
asymptotically fewer scalar operations than the classical algorithm, have been
considered primarily of theoretical interest.  Apart from Strassen's original
algorithm, few fast algorithms have been efficiently implemented or used in
practical applications.  However, there exist many practical alternatives to
Strassen's algorithm with varying performance and numerical properties.  Fast
algorithms are known to be numerically stable, but because their error bounds
are slightly weaker than the classical algorithm, they are not used even in
cases where they provide a performance benefit.

We argue in this paper that the numerical sacrifice of fast algorithms,
particularly for the typical use cases of practical algorithms, is not
prohibitive, and we explore ways to improve the accuracy both theoretically and
empirically.  The numerical accuracy of fast matrix multiplication depends on
properties of the algorithm and of the input matrices, and we consider both
contributions independently.  We generalize and tighten previous error analyses
of fast algorithms and compare their properties.  We discuss algorithmic
techniques for improving the error guarantees from two perspectives:
manipulating the algorithms, and reducing input anomalies by various forms of
diagonal scaling.  Finally, we benchmark performance and demonstrate our
improved numerical accuracy.
\end{abstract}


\pagestyle{myheadings}
\thispagestyle{plain}
\markboth{G. BALLARD, A. BENSON, A. DRUINSKY, B. LIPSHITZ, O. SCHWARTZ}{NUMERICAL STABILITY OF FAST MATRIX MULTIPLICATION}

\section{Introduction}
\label{sec:intro}

After Strassen's discovery of an algorithm for dense matrix-matrix
multiplication in 1969 \cite{strassen1969gaussian} that reduced the
computational complexity from the classical $O(N^3)$ (for multiplying two
$N\times N$ matrices) to $O(N^{\log_2{7}})$, there has been extensive effort to
understand fast matrix multiplication, based on algorithms with computational
complexity exponent less than 3.  From a theoretical perspective, there remains
a gap between the best known lower bound \cite{Landsberg14} and best known upper
bound \cite{Gall14} on the exponent.  From a practical perspective, it is
unlikely that the techniques for obtaining the best upper bounds on the exponent
can be translated to practical algorithms that will execute faster than the
classical one for reasonably sized matrices.  In this paper, we are interested
in the numerical stability of practical algorithms that have been demonstrated
to outperform the classical algorithm (as well as Strassen's in some instances)
on modern hardware \cite{benson2014framework}.

Nearly all fast matrix multiplication algorithms are based on recursion, using a
recursive rule that defines a method for multiplying matrices of fixed dimension
$\partdimm\times\partdimk$ by $\partdimk\times\partdimn$
$\partdimm\partdimk\partdimn$ scalar multiplications.  In this work, we use the
notation $\bc{\partdimm}{\partdimk}{\partdimn}$ for these algorithms.  For
practical algorithms, these fixed dimensions need to be very small, typically
$\partdimm,\partdimk,\partdimn<10$, as they define the factors by which the
dimensions of subproblems are reduced within the recursion.  Many such
algorithms have been recently discovered
\cite{benson2014framework,smirnov2013bilinear}.  Most fast algorithms share a
common bilinear structure and can be compactly represented by three matrices
that we denote by $\alg{\M U}{\M V}{\M W}$, following the notation of Bini and
Lotti~\cite{bini1980stability}.  Many key properties of the practicality of an
algorithm, including its numerical stability, can be derived quickly from its
$\alg{\M U}{\M V}{\M W}$ representation.  We also note that, because recursive
subproblems are again matrix multiplications, different recursive rules can be
combined arbitrarily.  Following the terminology of Ballard
\etal~\cite{BDHS11-SLB} and Demmel \etal~\cite{demmel2006fast}, we refer to
algorithms that vary recursive rules across different recursive levels and
within each level as \emph{non-uniform, non-stationary} algorithms.  If an
algorithm uses the same rule for every subproblem in each recursive level but
varies the rule across levels, we call it a \emph{uniform, non-stationary}
algorithm; those defined by only one rule are called \emph{stationary}
algorithms.  

Fast matrix multiplication is known to yield larger numerical errors than the
classical algorithm.  The forward error guarantee for the classical algorithm is
component-wise: the error bound for each entry in the output matrix depends only
on the dot product between the corresponding row and column of the input
matrices.  Fast algorithms perform computations involving other input matrix
entries that do not appear in a given dot product (their contributions
eventually cancel out), and therefore the error bounds for these algorithms
depend on more global properties of the input matrices.  Thus, fast algorithms
with no modification are known to exhibit so-called norm-wise
stability~\cite{bini1980stability} (sometimes referred to as Brent
stability~\cite{miller1975computational}) while the classical algorithm exhibits the stronger
component-wise stability, which is unattainable for fast algorithms
\cite{miller1975computational}.  

Our main goals in this paper are to explore means for improving the theoretical
error bounds of fast matrix multiplication algorithms as well as to test the
improvements with numerical experiments, focusing particularly on those
algorithms that yield performance benefits in practice.  For computing
$\M{C}=\M{A}\cdot\M{B}$, where $\M A$ is $M\times K$ and $\M B$ is $K\times N$,
norm-wise stability bounds for full recursion take the following form:
\begin{equation*}
\|\hat{\M{C}} - \M{C}\| \leq f_\text{alg}(K)  \|\M{A}\| \|\M{B}\|  \epsilon + O(\epsilon^2),
\end{equation*}
where $\|\cdot\|$ is the max-norm, $\epsilon$ is the machine precision, and
$f_\text{alg}$ is a polynomial function that depends on the algorithm
\cite{bini1980stability,demmel2006fast,higham2002accuracy}.\footnote{Here and elsewhere, the $O(\epsilon^2)$ term hides dependence on dimensions and norms of input matrices.}
For example,
$f_\text{alg}(K)=K^2$ for the classical algorithm, with no assumption on the
ordering of dot product computations.  We note that $f_\text{alg}$ is
independent of the input matrices, and $\|\M{A}\|\|\M{B}\|$ is independent of
the algorithm.  In this paper, we explore ways of improving each factor
separately.  Our main contributions include:
\begin{enumerate}
	\item generalizing and tightening previous error analysis of stationary fast algorithms to bound $f_\text{alg}$ in terms of the number of recursive steps used and two principal quantities derived from $\alg{\M U}{\M V}{\M W}$;
	\item presenting and comparing the stability quantities of recently discovered practical algorithms;
	\item exploring means of improving algorithmic stability through algorithm selection and non-uniform, non-stationary combination of algorithms;
	\item presenting diagonal scaling techniques to improve accuracy for inputs with entries of widely varying magnitudes; and
	\item showing empirical results of the effects of the various improvement techniques on both error and performance.
\end{enumerate}

The structure of the remainder of the paper is as follows.  We describe related
work in \cref{sec:rel_work} and introduce our notation for fast matrix
multiplication algorithms in \cref{sec:alg}.  \Cref{sec:bounds} presents the
error analysis for bounding $f_\text{alg}$ for general fast algorithms, and
\cref{sec:selection} discusses the implications of the bounds on known practical
algorithms.  We present diagonal scaling techniques in \cref{sec:scaling},
showing how to reduce the contribution of the input matrices to the error bound.
Finally, we discuss our results in \cref{sec:discussion} and provide
directions for improving implementations of fast matrix multiplication algorithms.

\section{Related Work}
\label{sec:rel_work}
Bini and Lotti \cite{bini1980stability} provide the first general error bound
for fast matrix multiplication algorithms, and their analysis provides the basis
for our results.  Demmel \etal~\cite{demmel2006fast} generalize Bini and Lotti's
results and show that all fast algorithms are stable.  A more complete summary
of the numerical stability of fast algorithms, with a detailed discussion of
Strassen's algorithm along with Winograd's variant, appears in Higham's
textbook~\cite[Chapter 23]{higham2002accuracy}.  We discuss these previous works
in more detail and compare them to our error bounds in \cref{sec:bounds}.

Castrapel and Gustafson~\cite{castrapel2007precision} and
D'Alberto~\cite{d2014better} discuss means of improving the numerical stability
of Strassen's algorithm (and Winograd's variant) using the flexibility of
non-uniform, non-stationary algorithms.  Castrapel and Gustafson propose general
approaches to such algorithms, and D'Alberto provides a specific improvement in
the case of two or more levels of recursion.

Smirnov~\cite{smirnov2013bilinear} describes strategies for discovering
practical fast algorithms and presents several new algorithms, including a
rank-23 algorithm for $\bc{3}{3}{3}$ with the fewest known nonzeros and an
algorithm for $\bc{6}{3}{3}$ that yields a better exponent than Strassen's.
Similar techniques are used by Benson and Ballard~\cite{benson2014framework},
and they demonstrate performance improvements over the classical and Strassen's
algorithm for both single-threaded and shared-memory multi-threaded
implementations.  Laderman \etal~\cite{LPS92} and later Kaporin
\cite{Kaporin99,kaporin2004aggregation} consider another form of practical
algorithms, ones that can achieve fewer floating point operations than the
Strassen-Winograd variant for certain matrix dimensions.  Kaporin demonstrates
better numerical stability than Strassen-Winograd and shows comparable
performance.  However, because the base case dimensions proposed are relatively
large (\eg, 13 or 20), we suspect that the performance will not be competitive
on today's hardware.  Further, because the $\alg{\M U}{\M V}{\M W}$
representations are not readily available, we do not consider these types of
algorithms in this work.

Dumitrescu \cite{dumitrescu1998improving} proposes a form of diagonal scaling to
improve the error bounds for Strassen's algorithm.  We refer to his approach as
\emph{outside scaling} and discuss it in more detail in \cref{sec:scaling}.
Higham \cite{higham2002accuracy} points out that inside scaling can also affect
the error bound but does not propose a technique for improving it.  Demmel
\etal~\cite{demmel2007fast} and Ballard \etal~\cite{ballard2012communication}
state (without proof) improved error bounds using either inside or outside
diagonal scaling.

\section{Fast Matrix Multiplication Algorithms}
\label{sec:alg}

Fast algorithms for matrix multiplication are those that perform fewer arithmetic operations than the classical algorithm in an asymptotic sense, achieving a computational complexity exponent less than 3 for the square case.
We consider such fast algorithms to be practical if they have been (or likely can be) demonstrated to outperform the most efficient implementations of the classical algorithm on current hardware~\cite{benson2014framework}.
From a practical viewpoint, because matrices arising in current applications have limited size, we can consider a fast algorithm's recursive rule being applied only a few times.
In light of this viewpoint, we state our algorithms (and error bounds) in terms of the number of recursive levels rather than the dimension of the base case, where the number of recursive levels need not be considered a fixed quantity.
In the rest of this section, we state the notational and terminology of the fast algorithms we consider in this paper.

\subsection{Base Case Algorithms}

A bilinear non-commutative algorithm that computes a product of an $\partdimm
\times \partdimk$ matrix and a $\partdimk \times \partdimn$ matrix ($\M{C} =
\M{A}\M{B}$) using $R$ non-scalar (active) multiplications is determined by a
$\partdimm\partdimk \times R$ matrix $\M{U}$, a $\partdimk\partdimn \times R$
matrix $\M{V}$, and a $\partdimm\partdimn \times R$ matrix $\M{W}$ such that
\begin{equation}\label{eqn:bilinear_alg}
\ME{C}{k} = \sum_{r=1}^{R}\ME{W}{kr}\ME{M}{r},
\quad \text{where} 
\quad \ME{M}{r} := \ME{S}{r} \cdot \ME{T}{r},
\quad \ME{S}{r} := \sum_{i=1}^{\partdimm\partdimk}\ME{U}{ir}\ME{A}{i},
\quad \ME{T}{r} := \sum_{j=1}^{\partdimk\partdimn}v_{jr}\ME{B}{j},
\end{equation}
for $k = 1, \ldots, \partdimm\partdimn$.  Here, the single indices of entries of
$\M A$ and $\M B$ assume column-major order, the single indices of entries of
$\M C$ assume row-major order, and $(\cdot)$ signifies an active multiplication.
We will refer to the dimensions of such an algorithm with the notation
$\bc{\partdimm}{\partdimk}{\partdimn}$, the rank of the algorithm by $R$, and
the set of coefficients that determine the algorithm with the notation $\alg{\M
  U}{\M V}{\M W}$.

\subsection{Stationary Algorithms}

Now we consider multiplying an $M \times K$ matrix $\M{A}$ by a $K \times N$ matrix $\M{B}$.
We will assume that $M$, $K$, and $N$ are powers of $\partdimm$, $\partdimk$, and $\partdimn$;
otherwise, we can always pad the matrices with zeros and the same analysis will hold.
The fast algorithm proceeds recursively by first partitioning $\M{A}$ into $\partdimm \times \partdimk$ submatrices of size $(M / \partdimm) \times (K / \partdimk)$ and $\M{B}$ into $\partdimk \times \partdimn$ submatrices of size $(K / \partdimk) \times (N / \partdimn)$ and then following \cref{eqn:bilinear_alg} by matrix blocks, \ie,
\begin{equation}
\label{eqn:recursive_alg}
\M{C}_{k} = \sum_{r=1}^{R}\ME{W}{kr}\M{M}_{r},
\; \text{where}
\; \M{M}_{r} := \M{S}_{r} \cdot \M{T}_{r},
\; \M{S}_{r} := \sum_{i=1}^{\partdimm\partdimk}\ME{U}{ir}\M{A}_{i},
\; \M{T}_{r} := \sum_{j=1}^{\partdimk\partdimn}v_{jr}\M{B}_{j}
\end{equation}
for $k = 1, \ldots, \partdimm\partdimn$, where $(\cdot)$ signifies a recursive
call to the algorithm.  Here, we are using single subscripts on matrices as an
index for the column- or row-major ordering of the matrix blocks.  The
algorithms in this class of fast matrix multiplication are called
\emph{stationary algorithms} because they use the same algorithm at each
recursive step \cite{demmel2006fast}.  However, we do not assume that stationary
algorithms recurse all the way to a base case of dimension 1; we assume only
that the base case computation (of whatever dimension) is performed using the
classical algorithm.  Thus, a stationary algorithm is defined by the triplet of
matrices $\alg{\M U}{\M V}{\M W}$ along with a number of recursive levels $L$
used before switching to the classical algorithm.

\subsection{Uniform, Non-Stationary Algorithms}

In contrast to the stationary algorithms discussed above,
\emph{uniform, non-stationary algorithms} employ a different fast algorithm, in the sense of \cref{eqn:bilinear_alg,eqn:recursive_alg}, at each recursive level \cite{BDHS11-SLB}.
The fast algorithm is the same at a given recursive level.
Specifically, we will consider uniform, non-stationary algorithms with $L$ steps of recursion, so the algorithm is specified by matrices $\M{U}^{[l]}$, $\M{V}^{[l]}$, $\M{W}^{[l]}$ of dimensions $\partdimm^{[l]}\partdimk^{[l]} \times R^{[l]}$, $\partdimk^{[l]}\partdimn^{[l]} \times R^{[l]}$, $\partdimm^{[l]}\partdimn^{[l]} \times R^{[l]}$, for $l = 1, \ldots, L$.

Uniform, non-stationary algorithms are of particular interest for maximizing
performance.  The fastest algorithm for a particular triplet of dimensions $M$,
$K$, and $N$ may depend on many factors; the same algorithm may not be optimal
for the recursive subproblems of smaller dimensions.  Assuming performance is
fixed for a given triplet of dimensions, the flexibility of non-stationary
algorithms allows for performance optimization over a given set of fast
algorithms.  However, in parallel and more heterogeneous settings, better
performance may be obtained by the greater generality of non-uniform,
non-stationary algorithms, described in the next section.

\subsection{Non-Uniform, Non-Stationary Algorithms}
\label{sec:nuns}

The final class of matrix multiplication algorithms we consider are {\it non-uniform, non-stationary algorithms}.
In contrast to the previous case, non-uniform, non-stationary algorithms use different algorithms within a single recursive level as well across recursive levels~\cite{BDHS11-SLB}, though we restrict the dimension of the partition by the fast algorithms at a given recursive level to be the same.
To define such algorithms, we must specify $\alg{\M U}{\M V}{\M W}$ for every node in the recursion tree, a total of $1+R^{[1]}+R^{[1]}R^{[2]}+\cdots+\prod_{l=1}^{L-1} R^{[l]}$ recursive rules.
We use superscript notation $[l,r_1,r_2,\dots,r_{l-1}]$ to denote a recursive node at level $l$, in the top-level subtree $r_1$, second level subtree $r_2$, and so on.

We demonstrate in \Cref{sec:non_uniform_non_stationary} that the flexibility of these algorithms allows for an improvement in the numerical stability of multi-level recursive algorithms.
We suspect that they also provide a performance benefit over stationary algorithms, though this has never been demonstrated empirically.

\section{Error Analysis}
\label{sec:bounds}

The work by Bini and Lotti~\cite{bini1980stability} provides the basic framework
for the forward error analysis of fast matrix multiplication algorithms.  They
provide general bounds for any square, stationary bilinear algorithm with
power-of-two coefficients (so that there is no error in scalar multiplications),
assuming that full recursion is used (a base case of dimension 1).  Demmel
\etal~\cite{demmel2006fast} extend the work of Bini and Lotti by (1)
accounting for errors induced by the scalar multiplications in bilinear
algorithms, (2) analyzing uniform, non-stationary bilinear fast matrix
multiplication algorithms (algorithms that use different fast matrix
multiplication routines at different levels of recursion), and (3) analyzing
group-theoretic fast matrix multiplication algorithms.  The bounds provided by
Demmel \etal~also assume square algorithms and that full recursion is
used.  Higham~\cite{higham2002accuracy} provides bounds for Strassen's original
algorithm as well as Winograd's variant in terms of the base case dimension
$n_0$, where the recursion switches to the classical algorithm.  Higham's bounds
are also slightly tighter (in the case of Strassen's and Winograd's algorithms)
than the general bounds previously mentioned.  We note that any matrix
multiplication algorithm satisfying the component-wise error bound must perform
at least $N^3$ arithmetic operations; that is, we cannot get the same
component-wise error bounds even when using just one step of recursion of a fast
algorithm~\cite{miller1975computational}.

The goal of the error analysis provided in this section is to generalize the
previous work in two main directions and to tighten the analysis particularly in
the case that nonzeros of $\M U$, $\M V$, and $\M W$ are not all $\pm1$.  First,
we will consider rectangular fast algorithms; that is, instead of considering
recursive rules for multiplying two $K\times K$ matrices, we consider the more
general set of rules for multiplying an $M\times K$ matrix by a $K\times N$
matrix.  Second, we will state our general bounds in terms of the number of
levels of recursion used.  Motivated by the results of recently discovered
practical algorithms~\cite{benson2014framework,smirnov2013bilinear}, we would
like to understand the theoretical error guarantees of an algorithm in terms of
its $\alg{\M U}{\M V}{\M W}$ representation.  The recent performance results
show that rectangular algorithms have practical value (even for multiplying
square matrices) and that, for performance reasons, typically only a small
number of recursive steps are used in practice.  Several recently discovered
practical algorithms include fractional power-of-two coefficients (e.g., $1/2$,
$1/4$, $1/8$), and we expect that other currently undiscovered useful algorithms
will include fractional coefficients that are not powers of two.  Therefore, we
make no assumptions on the entries of $\M U$, $\M V$, and $\M W$, and we derive
principal quantities below that can be tighter than the analogous quantities in
the previous works by Bini and Lotti~\cite{bini1980stability} and Demmel
\etal~\cite{demmel2006fast}, particularly in the case of fractional
coefficients.  This sometimes leads to much sharper error bounds (see
\cref{ex:better_emax}).

Finally, we point out that our representation of non-uniform, non-stationary algorithms is more convenient than previous work.
Careful choices of non-uniform, non-stationary algorithms have been shown to improve the numerical stability over stationary approaches (see \cref{ex:nunsStr}) \cite{d2014better}.
Bini and Lotti's bounds~\cite{bini1980stability} can be applied to such algorithms in terms of the global $\alg{\M U}{\M V}{\M W}$ representation, but the size of that representation grows quickly with the number of recursive levels.
Our representation (see \cref{sec:nuns}) and error bound (see \cref{sec:non_uniform_non_stationary}), given in terms of the base case rule used at each node in the recursion tree, allows for more efficient search of combinations of rules and has led to automatic discovery of more stable algorithms (see \cref{ex:nuns323}).

After defining the principal quantities of interest and specifying our model of computation, the rest of this section provides forward error bounds for each of the types of fast algorithms defined in \cref{sec:alg}.
We warn the reader that there are notational similarities and (sometimes subtle) inconsistencies with previous work, as a result of our tightening of the analysis.

\subsection{Principal quantities}
\label{sec:pq}

Following the approach of Bini and Lotti~\cite{bini1980stability}, we identify two principal quantities associated with a fast algorithm that, along with the dimensions of the algorithm and the number of levels of recursion, determine its theoretical error bounds.
The two principal quantities can be easily computed from the $\alg{\M U}{\M V}{\M W}$ representation, and we define them in terms of the following vectors:
\begin{equation}
\label{eqn:abg}
\VE{\alpha}{r} :=  \sum_{i=1}^{\partdimm\partdimk} \II(\ME{U}{ir}\neq0) \qquad 
\VE{\beta}{r} :=  \sum_{j=1}^{\partdimk\partdimn} \II(\ME{V}{jr}\neq0) \qquad 
\VE{\gamma}{k} :=  \sum_{r=1}^{R} \II(\ME{W}{kr}\neq0)
\end{equation}
\begin{equation}
\label{eqn:ab}
\VE{a}{r} := \sum_{i=1}^{\partdimm\partdimk} \vert\ME{U}{ir}\vert \qquad 
\VE{b}{r} :=  \sum_{j=1}^{\partdimk\partdimn} \vert\ME{V}{jr}\vert
\end{equation}
for $r=1,\dots,R$ and $k=1,\dots,\partdimm\partdimn$, where $\II$ is the
Boolean-valued indicator function with value 1 for true and 0 for false.  That
is, $\V \alpha$ is the vector of numbers of nonzeros in the columns of $\M U$,
$\V \beta$ is the vector of numbers of nonzeros in the columns of $\M V$, $\V
\gamma$ is the vector of numbers of nonzeros in the rows of $\M W$, $\V a$ is
the vector of column 1-norms of $\M U$, and $\V b$ is the vector of column
1-norms of $\M V$.  When $\M U$ and $\M V$ have $\pm 1$ entries, $\V \alpha = \V
a$ and $\V \beta = \V b$.

\begin{definition}
\label{def:prefactor}
The \emph{prefactor vector} $\V q$ is defined entry-wise by
\begin{equation}\label{eqn:qmax}
\VE{q}{k} =  \gamma_k + \max_r (\alpha_r + \beta_r) \II(\ME{W}{kr}\neq0)
\end{equation}
for $k=1,\dots,\partdimm\partdimn$, and the \emph{prefactor} $\qmax$ is defined as
$$\qmax = \max_k \VE{q}{k}.$$
\end{definition}

\begin{definition}
\label{def:stability}
The \emph{stability vector} $\V e$ is defined entry-wise by
\begin{equation}\label{eqn:emax}
\VE{e}{k} = \sum_{r=1}^{R} \VE{a}{r} \cdot \VE{b}{r} \cdot \vert\ME{w}{kr} \vert
\end{equation}
for $k=1,\dots,\partdimm\partdimn$, and the \emph{stability factor} $\emax$ is defined as
$$\emax = \max_k \VE{e}{k}.$$
\end{definition}
The principal quantities for several fast algorithms are listed in \cref{tab:principal_quantities}.

Bini and Lotti~\cite{bini1980stability} provide a definition of $\V q$ for two
different summation algorithms: sequential summation and serialized
divide-and-conquer (see \cref{sec:model}).  We choose the looser of these two
bounds (sequential summation) for generality and simpler notation.  However, our
results are easily converted to the tighter case.  Demmel \etal use the
serialized divide-and-conquer algorithm in their analysis.  Bini and Lotti's
analysis does not account for scalar (non-active) multiplication by elements of
$\M U$, $\M V$, and $\M W$, so their $\emax$ parameter depends only on the
non-zero structure, rather than the magnitude of the elements in these matrices
(cf. \cref{eqn:ab} and \cref{def:stability}).  Demmel \etal do account for the
multiplication by elements of $\M U$, $\M V$, and $\M W$.  However, their
$\emax$ parameter is identical to that of Bini and Lotti, and their bound
includes an additional factor of $(\| U \| \| V \| \| W \|)^{L}$, where $L$ is
the number of recursive levels and $\| \cdot \|$ is the max-norm.

\subsection{Model of arithmetic and notation}\label{sec:model}

We follow the notation of Demmel \etal~\cite{demmel2006fast}.
Let $\Theta = \{ \theta \;\mid\; \vert\theta\vert < \epsilon \}$ be the set of all errors bounded by $\epsilon$ (machine precision) and let $\Delta = \{ 1 + \theta \;\mid\; \theta \in \Theta \}$.
We assume the standard model of rounded arithmetic where the computed value of $op(a, b)$ is $op(a, b)(1 + \theta)$ for some $\theta \in \Theta$.
We use the set operation notation: $A + B := \{ a + b \mid a \in A, \; b \in B \}$, $A - B := \{ a - b \mid a \in A, \; b \in B \}$, and $A \cdot B := \{ a \cdot b \mid a \in A, \; b \in B \}$.

We define $A^j = A \cdot A\cdot \ldots \cdot A$ and note that $\Delta^j \subset \Delta^{j+1}$ as $1 \in \Delta$.
Furthermore, we will not distinguish between singleton sets and an element when using this notation, \eg, $op(a, b)(1 + \theta) \in op(a, b)\Delta$.
Finally, we will use the standard hat or $fl(\cdot)$ notation to denote a computed value, \eg, $\hat{\M{C}}$ or $fl(op(a, b)) \in op(a, b)\Delta$.

Under this arithmetic, the following fact for summation will be useful in our analysis
\begin{equation}\label{eqn:seq_summation_error}
fl\left(\sum_{i=1}^{N}fl(c_i \cdot a_i)\right) \in \left(\sum_{i=1}^{N}c_i \cdot a_i\right)\Delta^{N},
\end{equation}
where the algorithm for summation is simply to accumulate the terms $a_i$ one at a time, in sequential order.
By using a serialized divide-and-conquer summation, we can also achieve
\begin{equation}\label{eqn:dnc_summation_error}
fl\left(\sum_{i=1}^{N}fl(c_i \cdot a_i)\right) \in \left(\sum_{i=1}^{N}c_i \cdot a_i\right)\Delta^{1 + \lceil \log_2 N \rceil}.
\end{equation}
For generality, we will assume the more pessimistic bound in \cref{eqn:seq_summation_error}.
Our results can easily be modified for the error bounds in \cref{eqn:dnc_summation_error}.

We will also use the following property: 
\begin{equation}
\label{eqn:add_accum_err}
fl\lt(\sum_{i=1}^{N}c_i \Delta^{a_j}\right) \in \lt(\sum_{i=1}^{N}c_i \rt)\Delta^{N+\max_j a_j}.
\end{equation}

\subsection{Forward error analysis of stationary algorithms}
\label{sec:stationary}

The following theorem states the forward error bound for a stationary algorithm in terms of the principal quantities $\qmax$ and $\emax$ defined in \cref{sec:pq}, which can be readily determined from its $\alg{\M U}{\M V}{\M W}$ representation.
The sources of error are floating point error accumulation and possible growth in magnitude of intermediate quantities.
The floating point error accumulation depends in part on $\qmax$ and grows at worst linearly in $L$.
The growth of intermediate quantities depends on $\emax$ and grows exponentially in $L$, which typically dominates the bound.
\Cref{tab:principal_quantities} shows the values of these quantities for a variety of algorithms.

\begin{table}\begin{threeparttable}
\centering
\caption{ %
Principal quantities for a variety of fast matrix multiplication algorithms.
The rank of the algorithm ($R$) drives the asymptotic complexity, and the total number of non-zeroes in the $\M{U}$, $\M{V}$, and $\M{W}$ (nnz) affects the constant in the complexity. 
Likewise, the $\emax$ parameter drives the asymptotic behavior of the stability bound and the $\qmax$ parameter affects the constant (see \cref{thm:stationary_analysis}).
The stability exponent (stab.~exp.) denotes the asymptotic stability of the algorithm assuming square matrix multiplication (see \cref{eq:stabexp}), which allows for comparison of algorithms with different base case sizes.
}
\begin{tabular}{c c c c c c c c c }
\toprule
$\bc{\partdimm}{\partdimk}{\partdimn}$ & ref. & $\partdimm\partdimk\partdimn$ & $R$ & nnz & $\qmax$ & $\emax$ & stab.~exp. \\
\midrule
$\bc{2}{2}{2}$ & (classical) & 8 & 8 & 24 & 4 & 2 & 1 \\
$\bc{2}{2}{2}$ & \cite{strassen1969gaussian} & 8 & 7 & 36 & 8 & 12 & 3.58 \\
$\bc{3}{2}{2}$ & \cite{strassen1969gaussian}\tnotex{footnote:strassen} & 12 & 11 & 48 & 8 & 12 & 3.03 \\
$\bc{2}{3}{2}$ & \cite{strassen1969gaussian}\tnotex{footnote:strassen} & 12 & 11 & 48 & 9 & 13 & 3.03 \\
$\bc{4}{2}{2}$ & \cite{strassen1969gaussian}\tnotex{footnote:strassen} & 16 & 14 & 72 & 8 & 12 & 2.94 \\
$\bc{2}{4}{2}$ & \cite{strassen1969gaussian}\tnotex{footnote:strassen} & 16 & 14 & 72 & 12 & 24 & 2.94 \\
$\bc{3}{2}{3}$ & \cref{app:fast323} & 18 &15 & 94 & 10 & 20 & 3.21 \\
$\bc{3}{3}{2}$ & \cref{app:fast323} & 18 &15 & 94 & 11 & 23 & 3.21 \\
$\bc{3}{3}{3}$ & \cite{smirnov2013bilinear} & 27 & 23 & 139 & 15 & 29 & 3.07 \\
$\bc{4}{2}{3}$ & \cite{benson2014framework} & 24 & 20 & 130 & 14 & 34 & 3.38 \\
$\bc{3}{4}{2}$ & \cite{benson2014framework} & 24 & 20 & 130 & 14 & 30 & 3.38 \\
$\bc{2}{3}{4}$ & \cite{benson2014framework} & 24 & 20 & 130 & 14 & 35 & 3.38 \\
$\bc{4}{4}{2}$ & \cref{app:fast442} & 32 & 26 & 257 & 22 & 89 & 3.90 \\
$\bc{4}{2}{4}$ & \cref{app:fast442} & 32 & 26 & 257 & 23 & 92 & 3.93 \\
$\bc{3}{4}{3}$ & \cite{benson2014framework} & 36 & 29 & 234 & 23 & 100 & 3.66 \\
$\bc{3}{3}{4}$ & \cite{benson2014framework} & 36 & 29 & 234 & 18 & 71 & 3.66 \\
$\bc{3}{3}{6}$ & \cite{smirnov2013bilinear} & 54 & 40 & 960 & 39 & 428 & 4.69 \\
$\bc{3}{6}{3}$ & \cite{smirnov2013bilinear} & 54 & 40 & 960 & 48 & 728.5 & 4.69 \\
\bottomrule
\end{tabular}
\label{tab:principal_quantities}
\begin{tablenotes}
\item[*] \label{footnote:strassen} These algorithms correspond to straightforward generalizations of Strassen's $\bc{2}{2}{2}$ algorithm, either using two copies of the algorithm or one copy of the algorithm combined with the classical algorithm.
\end{tablenotes}
\end{threeparttable}\end{table}

\begin{theorem}
\label{thm:stationary_analysis}
Suppose that $\M{C} = \M{A} \cdot \M{B}$, where $\M{A} \in \mathbb{R}^{M \times K}$ and $\M{B} \in \mathbb{R}^{K \times N}$ is computed by using $L$ recursive steps of the fast matrix multiplication in \cref{eqn:recursive_alg}, with the classical algorithm used to multiply the $(M / \partdimm^L) \times (K / \partdimk^L)$ matrices by the $(K / \partdimk^L) \times (N / \partdimn^L)$ matrices at the base cases of the recursion.
Then the computed matrix $\hat{\M{C}}$ satisfies
\begin{equation*}\label{eqn:stationary_bound}
\| \hat{\M{C}} - \M{C} \| \le \lt(K / \partdimk^L + \qmax \cdot L\rt)(K / \partdimk^L) \cdot \emax^L\| \M{A} \| \| \M{B} \| \epsilon + O(\epsilon^2),
\end{equation*}
where $\| \cdot \|$ is the max-norm.
\end{theorem}

\begin{proof}
We begin by analyzing how relative errors propagate as we form the $\M{S}$ and $\M{T}$ matrices.
Let a superscript index in brackets denote a matrix formed at the specified level of recursion.
Following \cref{eqn:seq_summation_error}, we have the following error at the first recursive level:
\begin{equation*}
\hat{\M{S}}^{[1]}_r \in \sum_{i = 1}^{\partdimm\partdimk} \ME{U}{ir} \M{A}_i\Delta^{\alpha_r}, \quad \hat{\M{T}}^{[1]}_r \in \sum_{j = 1}^{\partdimk\partdimn} \ME{V}{jr} \M{B}_j\Delta^{\beta_r}, 
\end{equation*}
where $\V \alpha$ and $\V \beta$ are defined in \cref{eqn:abg}.

This error propagates as we recurse.
At the $l$th level of recursion, the inputs to the fast algorithm are given as sums of matrices $\M{A}_{\phi}$ and $\M{B}_{\psi}$, each with a possible error of $\Delta^{a}$ and $\Delta^{b}$, respectively, for some index sets $\phi$ and $\psi$ and some integers $a$ and $b$.
Following \cref{eqn:recursive_alg,eqn:seq_summation_error},
the algorithm simply accumulates an additional factor of $\Delta^{\alpha_{r_l}}$ and $\Delta^{\beta_{r_l}}$ before the matrices are passed to the subsequent level of recursion.
Thus, at the $L$th level of recursion, we have
\begin{equation}
\label{eqn:ST_error}
\hat{\M{S}}^{[L]}_r \in \M{S}^{[L]}_r\Delta^{\alpha_{r_1}+\cdots+\alpha_{r_L}}, \quad \hat{\M{T}}^{[L]}_r \in \M{T}^{[L]}_r\Delta^{\beta_{r_1}+\cdots+\beta_{r_L}}, 
\end{equation}
with $r=r_1+(r_2-1)R+\cdots+(r_L-1)R^{L-1}$.
Note that in exact arithmetic, 
\begin{equation}
\label{eqn:ST_exact}
\M{S}^{[L]}_r = \sum_{i=1}^{M_0^LK_0^L} \ME{U}{i_1r_1}\cdots\ME{U}{i_Lr_L} \M{A}_i, \quad \M{T}^{[L]}_r = \sum_{j=1}^{K_0^LN_0^L} \ME{V}{j_1r_1}\cdots\ME{V}{j_Lr_L} \M{B}_j,
\end{equation}
where $i=i_1+(i_2-1)M_0K_0+\cdots+(i_L-1)(M_0K_0)^{L-1}$ and $j=i_1+(j_2-1)K_0N_0+\cdots+(j_L-1)(K_0N_0)^{L-1}$ represent recursive orderings of the subblocks of $\M A$ and $\M B$.

We now use the classical algorithm to multiply the computed $\M{S}^{[L]}$ and $\M{T}^{[L]}$ matrices at the leaves of the recursion.
Because the inner dimension of each leaf-level matrix multiplication is $K/\partdimk^L$, from \cref{eqn:seq_summation_error,eqn:ST_error} we accumulate another factor of $\Delta^{K/\partdimk^L}$ to obtain
\begin{equation*}
\hat{\M{M}}^{[L]}_r \in \M{S}^{[L]}_r\M{T}^{[L]}_r\Delta^{\chi_r+K/\partdimk^L},
\end{equation*}
where $\chi_r=\alpha_{r_1}+\beta_{r_1}+\cdots+\alpha_{r_L}+\beta_{r_L}$ for $1\leq r \leq R^L$.

Next, the computed matrices $\M{M}^{[L]}$ are added to form $\M{C}$ following\cref{eqn:recursive_alg}.
At the $l$th level of recursion, sums of matrices $\M{M}^{[L]}_{\phi}$, for appropriate index sets $\phi$ and including accumulated error $\Delta^a$ for some integers $a$, are added together to form the intermediate computed quantities $\M{M}^{[l]}$. 
In the final step at the top of the recursion tree, we have
\begin{equation*}
\hat{\M{C}}_k \in \sum_{r=1}^{R} \ME{W}{kr}\hat{\M{M}}^{[1]}_{r}\Delta^{\gamma_k},
\end{equation*}
where $\V \gamma$ is as defined in \cref{eqn:abg}.
Following \cref{eqn:add_accum_err}, if $\hat{\M{M}}^{[1]}_{r} \in \M{M}^{[1]}_{r}\Delta^{x_r}$ for some integers $x_r$, then
\begin{equation*}
\hat{\M{C}}_k \in \sum_{r=1}^{R} \ME{W}{kr}\M{M}^{[1]}_{r}\Delta^{\gamma_k+\max_r x_r \cdot \II(\ME{W}{kr}\neq0)}.
\end{equation*}
Likewise, a factor of $\Delta^{\gamma_{k_l}}$ is accumulated at every recursive step, and the contributed error from the $\M{M}^{[L]}$ matrices comes from the leaf (that is involved in the summation) with maximum error.
Leaf matrix $\M{M}^{[L]}_r$ is involved in the summation for $\M{C}_k$ if $\ME{W}{k_1r_1}\cdots \ME{W}{k_Lr_L}\neq0$, where $r=r_1+(r_2-1)R+\cdots(r_L-1)R^{L-1}$ and $k=k_1+(k_2-1)M_0N_0+\cdots+(k_L-1)(M_0N_0)^{L-1}$.
Thus, we have
\begin{equation*}
\hat{\M{C}}_k \in \sum_{r=1}^{R^L} \ME{W}{k_1r_1}\cdots \ME{W}{k_1r_L} \M{M}^{[L]}_{r}\Delta^{\mu_k+\max_r \chi_r \cdot \II(\ME{W}{k_1r_1}\cdots \ME{W}{k_Lr_L}\neq0)+K/\partdimk^L},
\end{equation*}
where $\mu_k=\gamma_{k_1}+\cdots+\gamma_{k_L}$.

Let $\delta_k = \mu_k+\max_r \chi_r \cdot \II(\ME{W}{k_1r_1}\cdots \ME{W}{k_Lr_L}\neq0)+K/\partdimk^L$.  
In order to determine the largest accumulated error, we compute the maximum over all output blocks $\M{C}_k$:
\begin{align*}
\max_k \delta_k &= K/\partdimk^L + \max_{k_1,\dots,k_L} \lt\{\mu_k+ \max_{r_1,\dots,r_L} \chi_r \cdot \II(\ME{W}{k_1r_1}\cdots \ME{W}{k_Lr_L}\neq0) \rt\} \\
&= \begin{multlined}[t][.875\textwidth] K/\partdimk^L + \max_{k_1}  \lt\{ \gamma_{k_1}+ \max_{r_1} (\alpha_{r_1}+\beta_{r_1}) \II(\ME{W}{k_1r_1}\neq0) \rt\} + \cdots \\
 + \max_{k_L}  \lt\{ \gamma_{k_L}+ \max_{r_L} (\alpha_{r_L}+\beta_{r_L}) \II(\ME{W}{k_Lr_L}\neq0) \rt\} \end{multlined} \\
 &= K/\partdimk^L + \max_{k_1}  \lt\{ \gamma_{k_1}+ \max_{r_1} (\alpha_{r_1}+\beta_{r_1}) \II(\ME{W}{k_1r_1}\neq0) \rt\} \cdot L 
 = K/\partdimk^L + \qmax \cdot L,
\end{align*}
where $\qmax$ is given in \cref{def:prefactor}.

We now compute the forward error bound for each block of the output matrix.
We have $\M{E}_k=\M{C}_k - \hat{\M{C}}_k \in \sum_r \ME{W}{k_1r_1}\cdots \ME{W}{k_1r_L} \M{M}^{[L]}_{r} \Theta^{\delta_k}$, which implies (using \cref{eqn:ST_exact})
\begin{align*}
\lt\vert\M{E}_k\rt\vert &\leq \sum_{r=1}^{R^L} \lt\vert \ME{W}{k_1r_1}\cdots \ME{W}{k_1r_L} \M{S}^{[L]}_{r} \M{T}^{[L]}_{r} \rt\vert\delta_k \epsilon + O(\epsilon^2) \\
 &\leq \begin{multlined}[t][.875\textwidth] \sum_{r=1}^{R^L} \vert \ME{W}{k_1r_1}\cdots \ME{W}{k_1r_L}\vert \sum_{i=1}^{M_0^LK_0^L} \vert\ME{U}{i_1r_1}\cdots\ME{U}{i_Lr_L}\vert \vert\M{A}_i\vert \sum_{j=1}^{K_0^LN_0^L} \vert\ME{V}{j_1r_1}\cdots\ME{V}{j_Lr_L}\vert \vert\M{B}_j\vert\delta_k \epsilon \\
+ O(\epsilon^2) \end{multlined} \\
 &\leq \begin{multlined}[t][.9\textwidth] \sum_{r=1}^{R^L} \vert \ME{W}{k_1r_1}\cdots \ME{W}{k_1r_L}\vert \sum_{i=1}^{M_0^LK_0^L} \vert\ME{U}{i_1r_1}\cdots\ME{U}{i_Lr_L}\vert \sum_{j=1}^{K_0^LN_0^L} \vert\ME{V}{j_1r_1}\cdots\ME{V}{j_Lr_L}\vert \; \cdot \\ 
(K/K_0^L) \|\M{A}\|\|\M{B}\|\delta_k \epsilon + O(\epsilon^2). \end{multlined} 
\end{align*}
Let $\xi_k=\sum_r \vert \ME{W}{k_1r_1}\cdots \ME{W}{k_1r_L}\vert \sum_i \vert\ME{U}{i_1r_1}\cdots\ME{U}{i_Lr_L}\vert \sum_j \vert\ME{V}{j_1r_1}\cdots\ME{V}{j_Lr_L}\vert$.
In order to determine the largest intermediate quantity, we compute the maximum over all output blocks $\M{C}_k$:
\begin{align*}
\max_k \xi_k &= \max_{k_1,\dots,k_L} \sum_{r_1,\dots,r_L} \vert \ME{W}{k_1r_1}\cdots \ME{W}{k_Lr_L}\vert \sum_{i_1,\dots,i_L} \vert\ME{U}{i_1r_1}\cdots\ME{U}{i_Lr_L}\vert \sum_{j_1,\dots,j_L} \vert\ME{V}{j_1r_1}\cdots\ME{V}{j_Lr_L}\vert \\
 &= 
 \begin{multlined}[t][.875\textwidth]
 \lt(\max_{k_1} \sum_{r_1} \vert \ME{W}{k_1r_1}\vert \sum_{i_1} \vert\ME{U}{i_1r_1}\vert \sum_{j_1} \vert\ME{V}{j_1r_1}\vert \rt) \cdots
 \\
 \lt(\max_{k_L} \sum_{r_L} \vert \ME{W}{k_Lr_L}\vert \sum_{i_L} \vert\ME{U}{i_Lr_L}\vert \sum_{j_L} \vert\ME{V}{j_Lr_L}\vert \rt) \end{multlined} 
 \\
 &= \lt(\max_{k_1} \sum_{r_1} \vert \ME{W}{k_1r_1}\vert \sum_{i_1} \vert\ME{U}{i_1r_1}\vert \sum_{j_1} \vert\ME{V}{j_1r_1}\vert \rt)^L = \emax^L,
\end{align*}
where $\emax$ is given in \cref{def:stability}.

Computing $\max_k \vert\M{E}_k\vert$ by maximizing over $\delta_k$ and $\xi_k$ separately, we obtain our result.
We note that the two quantities may not achieve their maxima for the same $k$, but we ignore the possible looseness as the overall bound will typically be dominated by the value of $\emax$. 
\end{proof}

Note that if $L = \log_{\partdimk} K$ (full recursion), the bound in \cref{thm:stationary_analysis} becomes
\begin{equation*}
\| \hat{\M{C}} - \M{C} \| 
\le \lt(1 + \qmax \cdot L\rt) \cdot \emax^{\log_{\partdimk} K}\| \M{A} \| \| \M{B} \| \epsilon + O(\epsilon^2) \nonumber
\end{equation*}
which is the bound provided by Demmel \etal~\cite{demmel2006fast}, assuming $\partdimm = \partdimk = \partdimn$, $M=K=N$, all nonzeros of $\M U$ have the same value, all nonzeros of $\M V$ have the same value, and all nonzeros of $\M W$ have the same value.
If $L = 0$ (no recursion), we get the familiar bound
\begin{equation*}
\| \hat{\M{C}} - \M{C} \| \le K^2\| \M{A} \| \| \M{B} \| \epsilon + O(\epsilon^2).
\end{equation*}

\begin{example}\label{ex:better_emax}
Because our definition of $\emax$ (\cref{def:stability}) accounts for the
magnitude of the entries $\M U$, $\M V$ , and $\M W$ in situ, the bound from
\cref{thm:stationary_analysis} can be tighter than previous
analyses~\cite{bini1980stability,demmel2006fast} when $\M U$, $\M V$, or $\M W$
has entries outside of $\{-1, 0, 1\}$.  As an example, we consider a
$\bc{4}{4}{2}$ algorithm, where the $\M U$ and $\M W$ matrices have entries in
$\{-0.5, 0.5\}$~\cite{benson2014framework} (see \cref{app:fast442}).  For this
algorithm, $\emax$ according to \cref{def:stability} is $89$, while $\emax$
according to previous work is $125$.
\end{example}

\subsection{Forward error analysis of uniform, non-stationary algorithms}
\label{sec:non_stationary}

Recall that uniform, non-stationary algorithms use a single algorithm at each recursive level.
We denote the prefactor vector, stability vector, and partition dimensions of algorithm $\alg{\M{U}^{[l]}}{\M{V}^{[l]}}{\M{W}^{[l]}}$ at level $l$ by  
$\V q^{[l]}$, $\V e^{[l]}$, and $\partdimm^{[l]}$, $\partdimk^{[l]}$, and $\partdimn^{[l]}$.
Using a similar analysis to \cref{sec:stationary}, we get the following stability bound for this class of algorithms:
\begin{theorem}\label{thm:non_stationary_analysis}
Suppose that $\M{C} = \M{A} \cdot \M{B}$ is computed by a uniform, non-stationary algorithm with $L$ recursive steps of fast matrix multiplication, with the fast algorithm $\alg{\M{U}^{[l]}}{\M{V}^{[l]}}{\M{W}^{[l]}}$ used at level $l$ and the classical algorithm used to multiply the matrices at the base case of the recursion.
Then the computed matrix $\hat{\M{C}}$ satisfies
\begin{equation*}
\| \hat{\M{C}} - \M{C} \| \leq
\lt(\frac{K}{\prod_{l=1}^{L}\partdimk^{[l]}} + \sum_{l=1}^{L}\qmax^{[l]}\rt)\lt(\frac{K}{\prod_{l=1}^{L}\partdimk^{[l]}}\rt) \cdot
\lt(\prod_{l=1}^{L}\emax^{[l]}\rt)\| \M{A} \| \| \M{B} \| \epsilon + O(\epsilon^2).
\end{equation*}
\end{theorem}
\begin{proof}
The proof is similar to the proof of \cref{thm:stationary_analysis}.
The largest accumulation error $\delta$ now satisfies
\begin{align*}
\max_k \delta_k 
&= \begin{multlined}[t][.875\textwidth] \frac{K}{\prod_{l=1}^{L}\partdimk^{[l]}} + \max_{k_1}  \lt\{ \gamma^{[1]}_{k_1}+ \max_{r_1} (\alpha^{[1]}_{r_1}+\beta^{[1]}_{r_1}) \II(w^{[1]}_{k_1r_1}\neq0) \rt\} + \cdots \\
 + \max_{k_L}  \lt\{ \gamma^{[L]}_{k_L}+ \max_{r_L} (\alpha^{[L]}_{r_L}+\beta^{[L]}_{r_L}) \II(w^{[L]}_{k_Lr_L}\neq0) \rt\} 
= \frac{K}{\prod_{l=1}^{L}\partdimk^{[l]}} + \sum_{l=1}^{L}\qmax^{[l]},
 \end{multlined} 
\end{align*}
and the largest intermediate growth quantity $\xi$ satisfies
\begin{align*}
\max_k \xi_k 
 &= \begin{multlined}[t][.875\textwidth] \lt(\max_{k_1} \sum_{r_1=1}^{R^{[1]}} \vert \ME{W}{k_1r_1}^{[1]}\vert \sum_{i_1=1}^{\partdimm^{[1]}\partdimk^{[1]}} \vert\ME{U}{i_1r_1}^{[1]}\vert \sum_{j_1=1}^{\partdimk^{[1]}\partdimn^{[1]}} \vert\ME{V}{j_1r_1}^{[1]}\vert \rt) \cdots \\
 \lt(\max_{k_L} \sum_{r_L=1}^{R^{[L]}} \vert w^{[L]}_{k_Lr_L}\vert \sum_{i_L=1}^{\partdimm^{[L]}\partdimk^{[L]}} \vert\ME{U}{i_Lr_L}^{[L]}\vert \sum_{j_L=1}^{\partdimk^{[L]}\partdimn^{[L]}} \vert\ME{V}{j_Lr_L}^{[L]}\vert \rt) 
= \prod_{l=1}^{L}\emax^{[l]}.
 \end{multlined} \\
\end{align*}
\end{proof}

\subsection{Forward error analysis of non-uniform, non-stationary algorithms}
\label{sec:non_uniform_non_stationary}

We now consider non-stationary algorithms where the algorithm may be non-uniform at every given recursive level of fast matrix multiplication.
That is, at any node in the recursion tree, we may choose a different fast algorithm.
For simplicity, we assume that at level $l$ in the recursion tree, all algorithms have the same partitioning scheme and rank (so that the $\alg{\M{U}^{[l,r_1,\dots,r_{l-1}]}}{\M{V}^{[l,r_1,\dots,r_{l-1}]}}{\M{W}^{[l,r_1,\dots,r_{l-1}]}}$ representations have the same dimensions across all values $r_1,\dots,r_{l-1}$) and that after $L$ levels of recursion, all leaf nodes use the classical algorithm.

In the case of stationary algorithms, one $\alg{\M U}{\M V}{\M W}$ defines the entire algorithm; in the case of uniform non-stationary algorithms, $L$ choices of $\alg{\M{U}^{[l]}}{\M{V}^{[l]}}{\M{W}^{[l]}}$ define the entire algorithm; in this case, we have much more flexibility and can choose $1+R^{[1]}+R^{[1]}R^{[2]}+\cdots+\Pi_{l=1}^{L-1} R^{[l]}$ different fast algorithms (the number of internal nodes of the recursion tree).
Recall that we use the notation $[l,r_1,r_2,\dots,r_{l-1}]$ as a superscript to refer to the algorithm used at level $l$ in the recursion tree, where $r_1$ defines subtree membership at level 1, $r_2$ defines subtree membership at level 2, and so on, and $r_{l-1}$ defines the subtree node at the $l$th level.

Our analysis of these algorithms is fundamentally the same---we bound the accumulated error ($\delta$) and then bound the number of terms ($\xi$).
However, maximizing over all output blocks is not as straightforward and cannot be simplified as cleanly as in the previous cases.
In particular, we define the largest accumulation error $\delta$ recursively as $\max_k \delta_k^{[1]}$, where
\begin{align*}
\delta_k^{[1]} &= \frac{K}{\prod_{l=1}^{L}\partdimk^{[l]}} + \gamma_{k_1}^{[1]} + \max_{r_1} \delta^{[2,r_1]}_k \cdot \II(\ME{W}{k_1r_1}^{[1]} \neq 0), \\ 
\delta^{[2,r_1]}_k &= \gamma_{k_2}^{[2,r_1]} + \max_{r_2} \delta^{[3,r_1,r_2]}_k \cdot \II(\ME{W}{k_2r_2}^{[2,r_1]} \neq 0), \\ 
&\;\;\vdots\\
\delta^{[l,r_1,\dots,r_{l-1}]}_k &= \gamma_{k_l}^{[l,r_1,\dots,r_{l-1}]} + \max_{r_l} \delta^{[l+1,r_1,\dots,r_l]}_k \cdot \II(\ME{W}{k_lr_l}^{[l,r_1,\dots,r_{l-1}]} \neq 0), \\ 
&\;\;\vdots\\
\delta^{[L,r_1,\dots,r_{L-1}]}_k &= \gamma_{k_L}^{[L,r_1,\dots,r_{L-1}]} + \max_{r_L} \chi_r \cdot \II(\ME{W}{k_Lr_L}^{[L,r_1,\dots,r_{L-1}]} \neq 0), \text{ and}\\ 
\chi_r &= \alpha_{r_1}^{[1]} + \beta_{r_1}^{[1]} + \alpha_{r_2}^{[2,r_1]} + \beta_{r_2}^{[2,r_1]} + \cdots +  \alpha_{r_L}^{[L,r_1,\dots,r_{L-1}]} + \beta_{r_L}^{[L,r_1,\dots,r_{L-1}]}.
\end{align*}
This expression does not simplify as before. 
Note that for block $k$ of the output matrix, node $(r_1,\dots,r_{l-1})$ at level $l$ of the recursion tree accumulates error for the additions/subtractions required by matrix $\M{W}^{[l,r_1,\dots,r_{l-1}]}$ at that node plus the maximum accumulated error from any of the combined terms.  
The expression for $\chi_r$ reflects the number of additions and subtractions required to produce the factor matrices $\M{S}^{[L]}_r$ and $\M{T}^{[L]}_r$ at the leaf nodes, and the error accumulated during the classical matrix multiplications is included in the definition of $\delta_k^{[1]}$.

Likewise, the largest intermediate growth quantity $\xi$ is $\max_k \xi_k$, where
\begin{equation*}
\begin{multlined}[t][.95\textwidth]
\xi_k =\sum_{r_1,\dots,r_L} \lt\vert \ME{W}{k_1r_1}^{[1]}\ME{W}{k_2r_2}^{[2,r_1]}\cdots \ME{W}{k_Lr_L}^{[L,r_1,\dots,r_{L-1}]} \rt\vert \\
\cdot \sum_{i_1,\dots,i_L} \lt\vert\ME{U}{i_1r_1}^{[1]}\ME{U}{i_2r_2}^{[2,r_1]}\cdots\ME{U}{i_Lr_L}^{[L,r_1,\dots,r_{L-1}]} \rt\vert \cdot 
\sum_{j_1,\dots,j_L} \lt\vert\ME{V}{j_1r_1}^{[1]}\ME{V}{j_2r_2}^{[2,r_1]}\cdots\ME{V}{j_Lr_L}^{[L,r_1,\dots,r_{L-1}]} \rt\vert, 
\end{multlined} 
\end{equation*}
which we can simplify to
\begin{equation*}
\begin{multlined}[t][.95\textwidth]
\xi_k = \sum_{r_1} \lt\vert \ME{W}{k_1r_1}^{[1]} \rt\vert \VE{a}{r_1}^{[1]} \VE{b}{r_1}^{[1]} \cdot \sum_{r_2} \lt\vert \ME{W}{k_2r_2}^{[2,r_1]} \rt\vert \VE{a}{r_2}^{[2,r_1]} \VE{b}{r_2}^{[2,r_1]} \cdots  \\
\sum_{r_L} \lt\vert \ME{W}{k_Lr_L}^{[L,r_1,\dots,r_{L-1}]} \rt\vert \VE{a}{r_L}^{[L,r_1,\dots,r_{L-1}]} \VE{b}{r_L}^{[L,r_1,\dots,r_{L-1}]},
\end{multlined} 
\end{equation*}
where $\V{a}$ and $\V{b}$ vectors are defined as in \cref{eqn:ab}.
Note that we cannot simplify further as in the uniform case.

In Section~\ref{sec:searching_nuns}, we use non-uniform, non-stationary algorithms to
improve the numerical stability of fast matrix multiplication algorithms.

\section{Algorithm selection}
\label{sec:selection}

\cref{thm:stationary_analysis} immediately provides several options for
improving the numerical stability of fast matrix multiplication.  First, we can
look for algorithms with a smaller $\qmax$ and $\emax$.  Since prior work on
finding fast algorithms focuses on performance, this provides a new dimension
for algorithm design.  And in \cref{sec:search_stationary}, we compare several
stationary algorithms for the same base case as a first step in this dimension
of algorithm design.  We then extend this analysis to non-uniform,
non-stationary algorithms in \cref{sec:searching_nuns}.  Second, we can reduce
the number of recursive levels before using standard matrix multiplication
However, fewer recursive levels means an asymptotically slower algorithm.  We
examine this tradeoff in \cref{sec:recursive_levels}.  Finally, we can also
reduce $\| A \|$ and $\| B \|$ by pre- and post-processing the data, and we
provide several such strategies in \cref{sec:scaling}.

\subsection{Searching for better stationary algorithms}
\label{sec:search_stationary}

Typically, the only quantity of interest for finding fast matrix multiplication
algorithms is the rank of the solution, which controls the asymptotic
complexity.  However, we can also search for algorithms to minimize the $\qmax$
and $\emax$ quantities while maintaining the same rank.  This will improve the
numerical stability of the algorithm without sacrificing (asymptotic)
performance.  We will also consider the number of non-zeros (nnz) in the
solution, \ie, the sum of the number of non-zero entries in $\M{U}$, $\M{V}$,
and $\M{W}$, as this affects the constant in the asymptotic complexity and has
noticeable impact on empirical performance~\cite{benson2014framework}.  Thus,
the parameters of interest for these algorithms is a performance-stability
3-tuple (nnz, $\qmax$, $\emax$).  In general, the number of non-zeros is
positively correlated with $\qmax$ and $\emax$, since these quantities directly
depend on the non-zero patterns of $\M{U}$, $\M{V}$, and $\M{W}$ (see
\cref{eqn:qmax,eqn:emax}).

We first examined the base case $\bc{4}{2}{3}$, which has out-performed
Strassen's algorithm in practice~\cite{benson2014framework}.  We found 479
algorithms with rank $R = 20$ using numerical low-rank tensor decomposition
search techniques~\cite{benson2014framework}.  Of these, there were 208
performance-stability tuples.  The smallest nnz, $\qmax$, and $\emax$ quantities
over all algorithms were 130, 12, and 32, and the corresponding algorithms had
performance-stability tuples (130, 14, 34), (138, 12, 34), and (134, 13, 32). No
algorithm we found had parameters that achieved more than one of these minima,
so we call these three algorithms \emph{semi-optimal}.  Consequently, there is a
theoretical trade-off between performance and stability.  We note that although
this list of algorithms is not exhaustive, they are the only publicly available
$\bc{4}{2}{3}$ algorithms.%
\footnote{All of our algorithms, as well as the software for finding them, are
  publicly available at \url{https://github.com/arbenson/fast-matmul}.}

We tested the stability of these algorithms by computing the product of samples
of random matrices $\M{A} \in \mathbb{R}^{4096 \times 256}$ and $\M{B} \in
\mathbb{R}^{256 \times 2187}$.  The distributions were $\ME{A}{ij}$,
$\ME{B}{ij}~\sim~\text{Uniform}(0, 1)$ and $\ME{A}{ij}$,
$\ME{B}{ij}~\sim~\text{Uniform}(-1,1)$.  In addition to the three semi-optimal
algorithms described above, we also compared against an algorithm with a much
worse performance-stability tuple of (156, 26, 132), which we call a
\emph{sub-optimal} algorithm.  For each pair of matrices, we ran the four
algorithms with number of recursive levels $L = 1, 2, \ldots, 6$.  Our goal here
is to compare the errors of different algorithms with the same base case and
varying number of recursive levels---we are not claiming that any of these
algorithms are the best to use for these problem dimensions.

To estimate $\| \hat{\M{C}} -\M{C} \|$, we computed $\M{C}$ using the classical
algorithm in quadruple precision arithmetic.  All other computations used double
precision arithmetic.  Overall, we computed the errors for 100 random pairs
$\M{A}$ and $\M{B}$ for each distribution.  \cref{fig:errors423} reports the
maximum error over the 100 trials for each algorithm and each number of
recursive levels as well as the upper bound on the error from
\cref{thm:stationary_analysis}.  We see the following results:
\begin{enumerate}
\item %
The error bounds are still pessimistic, even with the improved analysis from
\cref{thm:stationary_analysis}.  Furthermore, the error bounds for the three
semi-optimal $\bc{4}{2}{3}$ algorithms are quite similar.
\item %
The true error increases with the number of recursive levels, as predicted by
\cref{thm:stationary_analysis} and modeled by the error bound.
\item %
For both distributions, the sub-optimal algorithm has larger errors than the
semi-optimal algorithms, as modeled by the error bound.
\item %
The difference between the semi-optimal algorithms depends on the matrices.  For
the $\text{Uniform}(0,1)$ distribution, there is a clear difference in error
between the algorithms.  Interestingly, the (134, 13, 32) semi-optimal algorithm
has larger errors than the (130, 14, 34), even though the former algorithm has
strictly better $Q$ and $E$ parameters.  For the $\text{Uniform}(-1, 1)$
distribution, the errors of the semi-optimal algorithms are practically
indistinguishable.
\end{enumerate}

We also considered the $\bc{2}{3}{2}$ base case, which has optimal rank $R =
11$~\cite{blaser2003complexity}.  One known algorithm that achieves the optimal
rank uses Strassen's algorithm on a $2 \times 2$ sub-block and classical matrix
multiplication on the remaining sub-blocks.  The base case of the algorithm is
small enough so that we could use a SAT solver~\cite{de2008z3} to find over
10,000 rank 11 $\bc{2}{3}{2}$ algorithms (ignoring symmetries).  We found that
the combination of Strassen's algorithm with the classical algorithm had a
strictly smaller performance-stability triple than all of the other rank 11
solutions.  We conclude that this algorithm is likely optimal in both a
performance and a stability sense for the class of $\bc{2}{3}{2}$ algorithms
where the scalar multiplications are $\pm 1$.

\begin{figure}
\includegraphics[width=6.4cm]{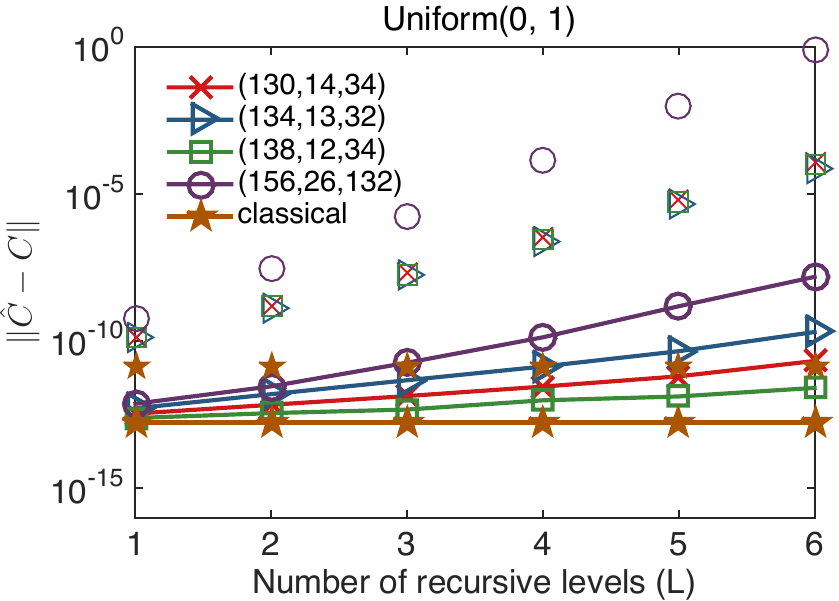}
\includegraphics[width=6.4cm]{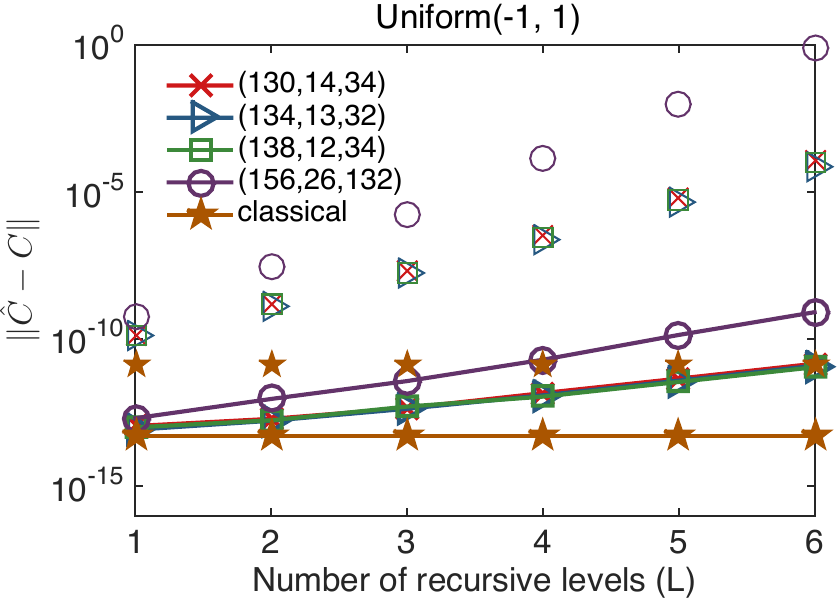}
\caption{%
Error for four $\bc{4}{2}{3}$ fast matrix multiplication algorithms with
different stability parameters and the classical algorithm as a function of the
number of recursive levels, $L$.  Three algorithms are semi-optimal in the sense
that they minimize one of the following quantities: number of non-zeros,
$\qmax$, or $\emax$.  The solid curves are the maximum experimental error over
100 pairs of random matrices, and the corresponding markers are the upper bounds
from \cref{thm:stationary_analysis}.  The experimental error increases with $L$,
as modeled by \cref{thm:stationary_analysis}.  The semi-optimal algorithms with
minimal nnz, $\qmax$, and $\emax$ all have similar performance, but the fast
algorithm with a worse performance-stability tuple is noticeably less stable in theory and practice.
}
\label{fig:errors423}
\end{figure}

\subsection{Searching for better non-uniform, non-stationary algorithms}
\label{sec:searching_nuns}

Stationary algorithms benefit from their simplicity, but non-uniform,
non-stationary algorithms provide a broader search space for algorithms with
better numerical properties.  We provide several examples below.

\begin{example}
\label{ex:nunsStr}
D'Alberto~\cite{d2014better} describes a non-uniform, non-stationary approach
using Strassen's algorithm that obtains a smaller stability factor than the
original stationary algorithm (for $L\geq2$).  Strassen's algorithm, with
$\alg{\M U}{\M V}{\M W}$ as given in \cref{app:strassen}, has stability vector
$e = \begin{bmatrix} 12 & 4 & 4 & 12 \end{bmatrix}$; two levels of recursion
with a stationary approach yields a two-level stability vector of $e\otimes e$
with maximum entry $12^2=144$.  D'Alberto shows that, for $L=2$, a stability
factor of 96 can be obtained with a non-uniform approach using one variant of
Strassen's algorithm.  One way to achieve this stability factor is to use the
alternative algorithm
\begin{equation*}
\alg{\tilde{\M U}}{\tilde{\M V}}{\tilde{\M W}} =
\alg{\M U}{
\lt(
\begin{bmatrix} 0 & 1 \\ 1 & 0 \end{bmatrix} \otimes
\begin{bmatrix} 1 & 0 \\ 0 & 1 \end{bmatrix}\rt) \M V
}
{\lt(
\begin{bmatrix} 1 & 0 \\ 0 & 1 \end{bmatrix} \otimes 
\begin{bmatrix} 0 & 1 \\ 1 & 0 \end{bmatrix}\rt)
\M W}
\end{equation*}
for nodes $[2,1]$, $[2,3]$, and $[2,4]$ of the recursion tree, while using the
original algorithm at nodes $[1]$, $[2,2]$, $[2,5]$, $[2,6]$, and $[2,7]$.
Similar improvements can be made based on the Strassen-Winograd algorithm, which
has a slightly larger stability factor.
\end{example}

A more generic non-uniform approach is described in a patent by Castrapel and
Gustafson~\cite{castrapel2007precision}.  They consider eight variants of the
Strassen-Winograd algorithm, defined by
\begin{equation*}
\alg{
\lt(
\begin{bmatrix} 0 & 1 \\ 1 & 0 \end{bmatrix}^x \otimes
\begin{bmatrix} 0 & 1 \\ 1 & 0  \end{bmatrix}^y\rt) 
\M U}
{\lt(
\begin{bmatrix} 0 & 1 \\ 1 & 0 \end{bmatrix} ^z\otimes
\begin{bmatrix} 0 & 1 \\ 1 & 0  \end{bmatrix}^x\rt)
\M V}{
\lt(
\begin{bmatrix} 0 & 1 \\ 1 & 0  \end{bmatrix}^y \otimes
\begin{bmatrix} 0 & 1 \\ 1 & 0 \end{bmatrix}^z\rt) \M W}
\end{equation*}
with $x,y,z\in\{1,2\}$.  The correctness of these variants can be derived from
the work of Johnson and McLoughlin~\cite[Equation (6)]{JM86}.  Castrapel and
Gustafson suggest using random, round-robin, or matrix-dependent selections of
algorithms to more evenly distribute the error, but they do not prove that any
particular techniques will reduce the stability factor.

\begin{example}
\label{ex:nuns323}
We can improve the two-level stability factor for the $\bc{3}{2}{3}$ case in a
similar manner.  The smallest stability factor we have discovered for this case
is $\emax=20$, given by the $\alg{\M U}{\M V}{\M W}$ in \cref{app:fast323},
which has stability vector
\[
e = \begin{bmatrix} 20 & 20 & 2 & 12 & 4 & 20 & 4 & 12 & 20 \end{bmatrix}.
\]
Compared to a uniform two-level stability factor of $20^2=400$, we can achieve a
stability factor of 352 using 3 variants of the algorithm.  We use the original
algorithm at nodes $[1]$, $[2,2]$, $[2,6]$, $[2,8]$, $[2,14]$, and $[2,15]$, the
variant
\begin{equation*}
\alg{\lt(I_2 \otimes
\begin{bmatrix} 1 & 0 & 0 \\ 0 & 0 & 1 \\ 0 & 1 & 0 \end{bmatrix}\rt)
\M U}{
\lt(
\begin{bmatrix} 1 & 0 & 0 \\ 0 & 0 & 1 \\ 0 & 1 & 0  \end{bmatrix} \otimes I_2\rt)
\M V}{
\lt( \begin{bmatrix} 1 & 0 & 0 \\ 0 & 0 & 1 \\ 0 & 1 & 0  \end{bmatrix} \otimes 
\begin{bmatrix} 1 & 0 & 0 \\ 0 & 0 & 1 \\ 0 & 1 & 0  \end{bmatrix}\rt)
\M W}
\end{equation*}
at nodes $[2,1]$, $[2,3]$, $[2,10]$, and $[2,11]$, and the variant
\begin{equation*}
\alg{\lt(I_2 \otimes \begin{bmatrix} 0 & 1 & 0 \\ 0 & 0 & 1 \\ 1 & 0 & 0  \end{bmatrix}\rt)
\M U}
{\lt(
\begin{bmatrix} 0 & 1 & 0 \\ 0 & 0 & 1 \\ 1 & 0 & 0  \end{bmatrix} \otimes I_2\rt)
\M V}
{\lt(
\begin{bmatrix} 0 & 1 & 0 \\ 0 & 0 & 1 \\ 1 & 0 & 0  \end{bmatrix} \otimes 
\begin{bmatrix} 0 & 1 & 0 \\ 0 & 0 & 1 \\ 1 & 0 & 0  \end{bmatrix}\rt)
\M W}
\end{equation*}
at nodes $[2,4]$, $[2,5]$, $[2,7]$, $[2,9]$, $[2,12]$, and $[2,13]$.  We suspect
that better two-level stability factors are achievable.
\end{example}

\subsection{Performance and stability trade-offs with a small number of recursive levels}
\label{sec:recursive_levels}

In addition to searching for better algorithms, we may also consider the effect
of the number of recursive levels on the numerical stability.  We now consider
the performance and stability of fast matrix multiplication algorithms across
several base cases and several values of $L$.  \cref{tab:principal_quantities}
summarizes the best known (to us) stability factors ($\emax$) for several
practical base case dimensions.  The columns of the table represent the relevant
performance and stability parameters for each algorithm, all of which can be
computed from the $\alg{\M U}{\M V}{\M W}$ representation.

The rank $R$ and the number of nonzeros (nnz), along with the number of
recursive levels used, determine the number of floating point operations
performed by the stationary version of the algorithm.  The rank can be compared
to the product $\partdimm\partdimk\partdimn$, the rank of the classical
algorithm for that base case.  The quantities $\qmax$ and $\emax$ are computed
using \cref{def:prefactor,def:stability}, respectively; for a given base case we
report the algorithm with the best known $\emax$ along with that algorithm's
$\qmax$.  We do not report both $\bc{\partdimm}{\partdimk}{\partdimn}$ and
$\bc{\partdimn}{\partdimk}{\partdimm}$ because the best algorithms for each have
identical nnz, $\emax$, and $\qmax$ parameters, due to transformations
corresponding to transposition of the matrix multiplication.

Although we stress that these algorithms will be used with only a few levels of
recursion, we also report the asymptotic stability exponent (stab.~exp.) in
order to compare algorithms across different base case dimensions.  If an
algorithm for a square base case $\bc{\partdimn}{\partdimn}{\partdimn}$ is used
on square matrices of dimension $N$ down to subproblems of constant dimension,
the bound of \cref{thm:stationary_analysis} can be simplified to
\begin{equation}
\label{eq:stabexp}
\| \hat{\M{C}} - \M{C} \| \le c\cdot N^{\log_{\partdimn}\emax}\log N\| \M{A} \| \| \M{B} \| \epsilon + O(\epsilon^2),
\end{equation}
where $c$ is a constant that depends in part on $\qmax$.  In this case, the
stability exponent is $\log_{\partdimn}\emax$.  We note that the first two rows
of \cref{tab:principal_quantities} match the results of Bini and
Lotti~\cite[Table 2]{bini1980stability}. The most stable rank-23 $\bc{3}{3}{3}$
algorithm of which we are aware is a cyclic rotation of the one given by
Smirnov~\cite{smirnov2013bilinear}.  In the case of rectangular base cases
$\bc{\partdimm}{\partdimk}{\partdimn}$, we assume a uniform, non-stationary
algorithm based on cyclic use of algorithms for
$\bc{\partdimm}{\partdimk}{\partdimn}$, $\bc{\partdimn}{\partdimm}{\partdimk}$,
and $\bc{\partdimk}{\partdimn}{\partdimm}$, where the three recursive rules are
transformations of each other, either by cyclic rotations or transposition (for
more details, see \cref{app:fast323,app:fast442}).

\cref{fig:stability_performance} shows the distribution of relative instability
and percentage of classical flops for the algorithms in
\cref{tab:principal_quantities}, for $L = 1, 2, 3, 4$.  We measure both terms
asymptotically.  Ignoring the quadratic cost of additions, the percentage of
classical flops is given by $(R / (M_0K_0N_0))^L$.  For large matrix dimension and $L$
small, we can ignore $\qmax$ by \cref{thm:stationary_analysis}, and we define the \emph{relative
instability} to be $(\emax / K_0^2)^{L}$, which is the factor by which the error bound exceeds that of the classical algorithm.  
In general, most algorithms follow a narrow 
log-linear trade-off between these two parameters.  However, there is still room
to select algorithms for a fixed number of recursion levels.  For example, with
$L=1$, the $\bc{3}{3}{3}$ algorithm has roughly the same stability and does
fewer floating point operations than Strassen's algorithm.

\begin{figure}
\centering
\includegraphics[width=10cm]{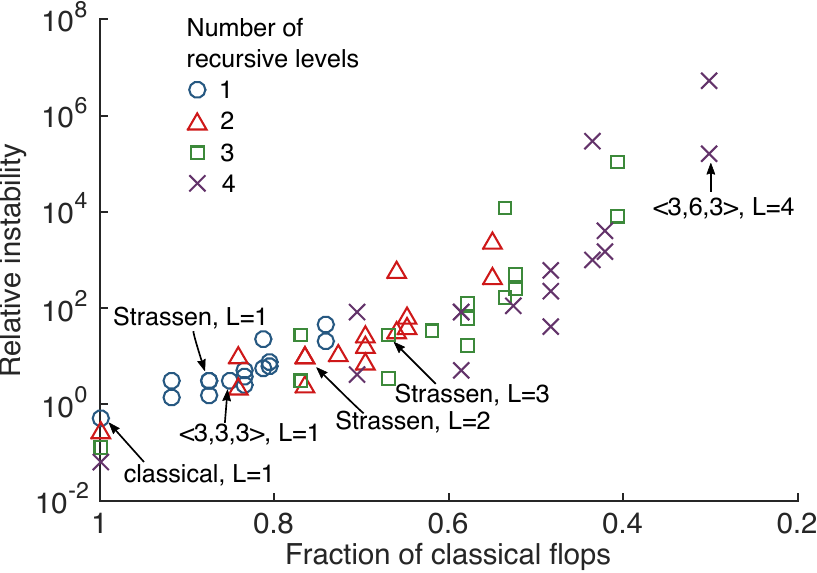}
\caption{%
Distribution of relative instability, $(E / K_0^2)^L$, and percentage of
classical flops, $(R / (M_0K_0N_0))^L$, for
the algorithms in \cref{tab:principal_quantities} with $L = 1, 2, 3, 4$.
A larger value on the y-axis means a less stable algorithm, and
a smaller value on the x-axis means a faster algorithm for large problem sizes.
There is a general log-linear trade-off between stability and number of
floating point operations.
}
\label{fig:stability_performance}
\end{figure}

\section{Scaling}
\label{sec:scaling}

We now turn our attention to strategies for pre- and post-processing matrices in
order to improve numerical stability.  The error bounds from \cref{sec:bounds}
can be summarized by the following element-wise absolute error bound 
\begin{equation}\label{eqn:abs_error}
\vert \ME{C}{ij} - \hat{c}_{ij} \vert \le f_\text{alg}(K) \| A \| \| B \| \epsilon + O(\epsilon^2).
\end{equation}
Recall that $f_\text{alg}$ is the (at worst) polynomial function of the inner dimension that depends on the particular algorithm used.
Unfortunately, these bounds can often be quite large when $|c_{ij}|$ is small relative to $\|A\|\|B\|$.
The purpose of this section is to address the contribution of $\| A \|$ and  $\| B \|$ to the error bound, ignoring the particular fast algorithm that is used. 
Thus, for the remainder of this section, we will ignore $f_\text{alg}$ and consider it a fixed quantity, so that $f_\text{alg}(K)\epsilon = O(\epsilon)$, and we will  focus on relative error. 

The following example shows that the relative error from fast matrix multiplication computations can be large.
We note that for the purposes of this example and subsequent examples throughout the paper, we assume that floating point operations incur an
$\epsilon$ relative error even if the operands happen to be powers of two or if we are subtracting identical values.  This assumption does not limit the generality of our examples; instead, we have chosen the matrix entries to make the examples
as simple as possible.
One could apply small independent relative perturbations on matrix entries
for the examples to work in standard floating point arithmetic.
\begin{example}\label{ex:out_scaling}
Consider the matrices
\begin{equation}\label{eqn:out_scaling_ex}
\M{A} =
\begin{bmatrix}
1 & 1 \\
1 & 1\\
\end{bmatrix},\;
\M{B} =
\begin{bmatrix}
\smallentry & 1 \\
\smallentry & 1 \\
\end{bmatrix},\;
\M{C} = \M{A} \cdot \M{B} =
\begin{bmatrix}
2\smallentry & 2 \\
2\smallentry & 2 \\
\end{bmatrix},
\end{equation}
for small $\smallentry > 0$.  By \cref{eqn:abs_error}, we have the following
relative error bound
\begin{equation}\label{eqn:scaling_motivation_ex}
\frac{\vert\ME{C}{11} - \hat{c}_{11}\vert}{\vert\ME{C}{11}\vert} \le O\lt(\frac{\| \M{A} \| \| \M{B} \|\epsilon}{\vert \ME{C}{11}\vert}\rt) = O\lt(\epsilon / \smallentry\rt),
\end{equation}
which can be quite large for small $\smallentry$.  Furthermore, this bound is
actually achieved with Strassen's algorithm (see \cref{app:strassen} for the
definition of Strassen's algorithm).  Specifically, Strassen's algorithm
computes

\begin{align*}\label{eqn:strassen_out_no_scaling}
m_1 &= (\ME{A}{11} + \ME{A}{22})(\ME{B}{11} + \ME{B}{22}) = (1 + 1) \cdot (\smallentry + 1) \\
m_4 &= \ME{A}{22}(\ME{B}{21} - \ME{B}{11}) = 1 \cdot (\smallentry - \smallentry)  \\
m_5 &= (\ME{A}{11} + \ME{A}{12})\ME{B}{22} = (1 + 1) \cdot 1 \\
m_7 &= (\ME{A}{12} - \ME{A}{22})(\ME{B}{21} + \ME{B}{22}) = (1 - 1) \cdot (\smallentry + 1)  \\
\ME{C}{11} &= m_1 + m_4 - m_5 + m_7.
\end{align*}

There are terms of size $O(1)$ in computing $m_1$, $m_4$, $m_5$, and $m_7$, so
the absolute error $\vert \ME{C}{11} - \hat{c}_{11} \vert$ is $O(\epsilon)$.
Since $\ME{C}{11} = \smallentry$, the relative error is $O(\epsilon / \smallentry)$.
\end{example}

We now demonstrate several methods for improving numerical stability issues by
pre-processing $\M{A}$ and $\M{B}$ and post-processing $\M{C}$.  The idea
underlying these methods is the following straightforward observation
\begin{equation}\label{eqn:scaling}
\M{C} = \leftscalemat\leftscalemat^{-1} \M{A}\inscalemat\inscalemat^{-1}\M{B}\rightscalemat^{-1}\rightscalemat,
\end{equation}
for any nonsingular scaling matrices $\leftscalemat$, $\rightscalemat$, and
$\inscalemat$.  By taking advantage of the associativity of matrix
multiplication (in exact arithmetic) and scaling matrices $\leftscalemat$,
$\rightscalemat$, and $\inscalemat$ that are easy to apply, we can improve the
norm-wise bound in \cref{eqn:stationary_bound} without significantly affecting
the performance of the algorithm.

For the algorithms and analysis in this section, we will consider diagonal
scaling matrices with positive diagonal entries.  In order to simplify the
analysis, we will assume that there is no numerical error in applying or
computing the scaling matrices.  This could be achieved, for example, by
rounding the scaling matrix entries to the nearest power of two.  Regardless,
the error introduced by the fast matrix multiplication algorithm has the larger
impact on the stability, and the scaling matrices can curb numerical
inaccuracies.

\subsection{Outside scaling}

In light of \cref{eqn:scaling}, Dumitrescu proposed the following outside scaling matrices~\cite{dumitrescu1998improving}:
\begin{equation*}
\leftscalemat = \diag{\max_{j} \vert\ME{A}{ij}\vert},\quad
\rightscalemat = \diag{\max_{i} \vert\ME{B}{ij}\vert}.
\end{equation*}

The resulting procedure is \cref{alg:out_scaling}.

\begin{algorithm}[H]
\begin{algorithmic}[1]
\Require matrices $\M{A}$ and $\M{B}$ 
\Ensure $\M{C} = \M{A} \cdot \M{B}$
\State $\leftscalemat \leftarrow \diag{\max_{j} \vert\ME{A}{ij}\vert}$ \;
	\State $\M{A}' \leftarrow \leftscalemat^{-1}\M{A}$ \;
	\State $\rightscalemat \leftarrow \diag{\max_{i} \vert\ME{B}{ij}\vert}$ \;
	\State $\M{B}' \leftarrow \M{B}\rightscalemat^{-1}$ \;
	\State $\M{C}' \leftarrow \M{A} \cdot \M{B}$ with fast matrix multiplication. \;
	\State $\M{C} \leftarrow \leftscalemat\M{C}'\rightscalemat$ \;
\end{algorithmic}
\caption{Outside scaling for fast matrix multiplication}
\label{alg:out_scaling}
\end{algorithm}

Clearly, the algorithm correctly computes $\M{C} = \M{A} \cdot \M{B}$ in exact
arithmetic, provided there are no all-zero rows in $\M{A}$ or all-zero columns
in $\M{B}$.  Importantly, the norm-wise bound in \cref{eqn:stationary_bound}
applies to the scaled matrices $\M{A}'$ and $\M{B}'$.  In particular, we get the
following improved bound~\cite{dumitrescu1998improving}.
\begin{proposition}\label{prop:out_scaling_bound}
Using \cref{alg:out_scaling},
\begin{equation*}
\vert \ME{C}{ij} - \hat{c}_{ij} \vert \le O(\epsilon) \| \ME{A}{i, :} \| \| \ME{B}{:, j} \|.
\end{equation*}
\end{proposition}
\begin{proof}
Outside scaling ensures that $\| \M{A}' \| = \| \M{B}' \| = 1$, so by
\cref{eqn:abs_error}, $\| \M{C}' - \hat{\M{C}}' \| \le O(\epsilon)$.  Since
$\M{C}' - \hat{\M{C}}' = \leftscalemat(\hat{\M{C}} - \M{C})\rightscalemat$, the
result follows from the fact that the $i$th diagonal entry of $\leftscalemat$ is
$\| \ME{A}{i, :} \|$ and $j$th diagonal entry of $\rightscalemat$ is
$\|\ME{B}{:, j} \|$.
\end{proof}\smallspace

For the matrices in \cref{ex:out_scaling}, the bound from
\cref{prop:out_scaling_bound} improves upon \cref{eqn:scaling_motivation_ex}:

\begin{equation*}
\frac{\vert\ME{C}{11} - \hat{c}_{11}\vert}{\vert\ME{C}{11}\vert} \le O\lt(\frac{\| \ME{A}{1,:} \| \| \ME{B}{:,1} \|\epsilon}{\vert \ME{C}{11}\vert}\rt) = O\lt(\epsilon \rt)
\end{equation*}

This indeed improves the numerical stability of Strassen's algorithm.
For the matrices in \cref{eqn:out_scaling_ex}, the outside scaling is
\begin{equation}
\M{A}' \leftarrow
\begin{bmatrix}
1 & 1 \\
1 & 1 \\
\end{bmatrix},\;
\M{B}' \leftarrow
\begin{bmatrix}
1 & 1 \\
1 & 1 \\
\end{bmatrix},\;
\M{C}' = \M{A}' \cdot \M{B}' = 
\begin{bmatrix}
2 & 2 \\
2 & 2 \\
\end{bmatrix}. \nonumber
\end{equation}

And when computing $\M{C}'$ with Strassen's algorithm,
\begin{align*}\label{eqn:strassen_outside_scaling}
m_1 &= (\ME{A'}{11} + \ME{A'}{22})(\ME{B'}{11} + \ME{B'}{22}) = (1 + 1) \cdot (1 + 1) \\
m_4 &= \ME{A'}{22}(\ME{B'}{21} - \ME{B'}{11}) = 1 \cdot (1 - 1) \\
m_5 &= (\ME{A'}{11} + \ME{A'}{12})\ME{B'}{22} = (1 + 1) \cdot 1 \\
m_7 &= (\ME{A'}{12} - \ME{A'}{22})(\ME{B'}{21} + \ME{B'}{22}) = (1 - 1) \cdot (1 + 1) \\
\ME{C'}{11} &= m_1 + m_4 - m_5 + m_7.
\end{align*}

Now, all sub-terms are on the order of unity, so the relative error in computing $\ME{C}{11}'$ is $O(\epsilon)$.

\subsection{Inside scaling}

There are pairs of matrices where outside scaling is not sufficient for numerical stability.
\begin{example}\label{ex:in_scaling}
Consider the matrices
\begin{equation}\label{eqn:in_scaling_ex1}
\M{A} =
\begin{bmatrix}
1 & \smallentry \\
1 & \smallentry \\
\end{bmatrix},\;
\M{B} =
\begin{bmatrix}
\smallentry & \smallentry \\
1 & 1  \\
\end{bmatrix},\;
\M{C} = \M{A}\cdot\M{B} =
\begin{bmatrix}
2\smallentry & 2\smallentry \\
2\smallentry & 2\smallentry  \\
\end{bmatrix}.
\end{equation}
for small $\smallentry > 0$.
Using outside scaling on these matrices does nothing since $\leftscalemat = \rightscalemat = \M{I}$.
However, using a fast algorithm can still have severe numerical stability issues.
Computing $\ME{C}{12}$ with Strassen's algorithm uses the following computations:
\begin{align*}\label{eqn:strassen_in_scaling1}
m_3             &= \ME{A}{11}(\ME{B}{12} - \ME{B}{22}) = 1 \cdot (\smallentry - 1) \\
m_5             &= (\ME{A}{11} + \ME{A}{12})\ME{B}{22} = (1 + \smallentry) \cdot 1 \\
\ME{C}{12} &= m_3 + m_5. 
\end{align*}
The computation of $m_3$ and $m_5$ has terms of unit size,
so $\vert \ME{C}{12} - \hat{c}_{12} \vert$ is $O(\epsilon)$ and the relative error is $O(\epsilon / \smallentry)$.
This is reflected in the bound from \cref{eqn:abs_error}:
\begin{equation*}\label{eqn:outside_scaling_bound}
\| \M{A} \| \| \M{B} \| / \vert\ME{C}{12}\vert = 1/ (2\smallentry).
\end{equation*}

\end{example}

We now propose a technique called \emph{inside scaling} based on the following matrix:
\begin{equation}
\inscalemat = \diag{ \sqrt{\frac{\max_j \vert \ME{B}{kj} \vert}{\max_i \vert \ME{A}{ik} \vert}} } \label{eqn:inside_scaling_D}
\end{equation}
The resulting procedure is in \cref{alg:in_scaling}.  The idea is to scale the
columns of $\M{A}$ and the corresponding rows of $\M{B}$ to have the same norm.
In general, we get an improved error bound, as detailed in
\cref{prop:in_scaling_bound}.  We note that there exist several references to
inside scaling \cite{ballard2012communication, demmel2007fast,
  higham2002accuracy, lipshitz2012communication}, though to our knowledge this
is the first explicit statement of the diagonal values in \cref{eqn:inside_scaling_D}.
A crude inside scaling method was proposed earlier by Brent~\cite{brent1970error}, where the inside scaling matrix is $\M{D} = \sqrt{\| \M{B} \| / \| \M{A} \|} \M{I}$.

\begin{algorithm}
\begin{algorithmic}[1]
\Require{matrices $\M{A}$ and $\M{B}$}
\Ensure{$\M{C} = \M{A} \cdot \M{B}$}
	\State $\inscalemat \leftarrow \diag{ \sqrt{\frac{\max_j \vert \ME{B}{kj} \vert}{\max_i \vert \ME{A}{ik} \vert}} }$ \;
	\State $\M{A}' \leftarrow \M{A}\inscalemat$ \;
	\State $\M{B}' \leftarrow \inscalemat^{-1}\M{B}$ \;
	\State $\M{C} \leftarrow \M{A}' \cdot \M{B}'$ with fast matrix multiplication. \;
\end{algorithmic}
\caption{Inside scaling for fast matrix multiplication}
\label{alg:in_scaling}
\end{algorithm}

\begin{proposition}\label{prop:in_scaling_bound}
Using \cref{alg:in_scaling},
\begin{equation*}
\| \hat{\M{C}} - \M{C} \| \le O(\epsilon) \max_{i,k,j} \vert\ME{A}{ik}\vert \vert\ME{B}{kj}\vert.
\end{equation*}
\end{proposition}
\begin{proof}
By \cref{eqn:abs_error},
\begin{equation}
\| \hat{\M{C}} - \M{C} \| \le O(\epsilon)\| \M{A} \inscalemat \| \| \inscalemat^{-1} \M{B} \| 
= O(\epsilon)\left(\max_{k} \| \ME{A}{:, k} \| \ME{D}{kk} \right)\left(\max_{l}\ME{D}{ll}^{-1} \| \ME{B}{l, :} \| \right). \nonumber
\end{equation}
By the definition of $\inscalemat$,
\begin{equation}
\| \ME{A}{:, k} \| \ME{D}{kk} = \ME{D}{kk}^{-1} \| \ME{B}{k, :} \| = \sqrt{\| \ME{A}{:, k} \| \| \ME{B}{k, :} \|}, \nonumber
\end{equation}
so the two maxima are attained at the same index.
The result then follows from the fact that $\| \ME{A}{:, k} \| \| \ME{B}{k, :} \| = \max_{i, k, j} \vert \ME{A}{ik} \vert \vert \ME{B}{kj} \vert$.
\end{proof}\smallspace

For $\M{A}$ and $\M{B}$ in \cref{ex:in_scaling}, $\max_{i,k,j}
\vert\ME{A}{ik}\vert \vert\ME{B}{kj}\vert = \smallentry$, and we get an
$O(\epsilon)$ relative error bound for computing each entry in $\M{C}$.  The
inside scaling updates to the matrices in \cref{eqn:in_scaling_ex1} are
\begin{equation*}
\inscalemat =
\begin{bmatrix}
\sqrt{\smallentry} & 0 \\
0 & 1 / \sqrt{\smallentry}
\end{bmatrix},\;
\M{A}' \leftarrow
\begin{bmatrix}
\sqrt{\smallentry} & \sqrt{\smallentry} \\
\sqrt{\smallentry} & \sqrt{\smallentry}
\end{bmatrix},\;
\M{B}' \leftarrow
\begin{bmatrix}
\sqrt{\smallentry} & \sqrt{\smallentry}\\
\sqrt{\smallentry} & \sqrt{\smallentry}
\end{bmatrix}.
\end{equation*}

Strassen's algorithm now computes
\begin{align*}\label{eqn:strassen_in_scaling2}
m_3             &= \ME{A}{11}'(\ME{B}{12}' - \ME{B}{22}') = \sqrt{\smallentry} \cdot (\sqrt{\smallentry} - \sqrt{\smallentry}) \\
m_5             &= (\ME{A}{11}' + \ME{A}{12}')\ME{B}{22}' = (\sqrt{\smallentry} + \sqrt{\smallentry}) \cdot \sqrt{\smallentry} \\
\ME{C}{12}' &= m_3 + m_5.
\end{align*}
This time, the computation of $m_3$ and $m_5$ involves terms on the order of $\smallentry$ instead of on the order of unity, and we get $O(\epsilon)$ relative error in the computation.

\subsection{Repeated outside-inside scaling}
\label{sec:alternating-scaling}

Next, we consider repeatedly applying outside and inside scaling in alternating order, as shown in \cref{alg:alt_scaling}. 
This process can only improve the error bounds, and it eventually converges.
Outside and inside scaling can simply be applied several times, or the user can specify a cheaply computed stopping criterion that will guarantee a relative distance from the limit point.

\begin{algorithm}
\caption{Repeated outside-inside scaling for fast matrix multiplication}
\label{alg:alt_scaling}
\begin{algorithmic}[1]
    \algblockdefx{Alternate}{EndAlternate}{\textbf{alternate}}{\textbf{until converged}}
    \algblockdefx{Step}{EndStep}[1]{\textbf{#1 step}}{\textbf{end}}
    \State $\M{A}' \gets \M{A}$, $\M{B}' \gets \M{B}$, $\leftscalemat \gets \M{I}$, $\rightscalemat \gets \M{I}$
    \Alternate
        \Step{O}
            \State $\leftscalemat' \gets \diag{\max_{k}\abs{\ME{A'}{ik}}}$
            \State $\leftscalemat \gets \leftscalemat \leftscalemat'$
            \State $\M{A}' \gets (\leftscalemat')^{-1} \M{A}'$
            \State $\rightscalemat' \gets \diag{\max_{k}\abs{\ME{B'}{kj}}}$
            \State $\rightscalemat \gets \rightscalemat' \rightscalemat$
            \State $\M{B}' \gets \M{B}' (\rightscalemat')^{-1}$
        \EndStep
        \Step{I}
        \State $\inscalemat \gets \diag{\! \sqrt{\frac{\max_{j}\abs{\ME{B'}{kj}}}{\max_i\abs{\ME{A'}{ik}}}} \,}$
            \State $\M{A}' \gets \M{A}' \inscalemat$
            \State $\M{B}' \gets \inscalemat^{-1} \M{B}'$
        \EndStep
    \EndAlternate
    \State $\M{C}' \gets \M{A}' \cdot \M{B}'$ with fast matrix multiplication
    \State $\M{C} \gets \leftscalemat \M{C}' \rightscalemat$
\end{algorithmic}
\end{algorithm}

We start with our accuracy analysis. In our analysis, we use $\At$ and $\Bt$ to
denote the values of $\M{A}'$ and $\M{B}'$, respectively, after $t$ steps of
\cref{alg:alt_scaling}. We also use $\VnE{r}{t}{i}$ and $\VnE{s}{t}{j}$ to
denote the diagonal elements of $\leftscalemat$ and $\rightscalemat$,
respectively, after $t$ steps. The initial values of these variables correspond
to $t = 0$ in our notation.

\begin{proposition}
\label{prop:alt-scaling}
Let $t$ be the number of steps of \cref{alg:alt_scaling} that we complete. The computed product satisfies
\begin{equation*}
    \abs{\ME{C}{ij} - \hat{c}_{ij}}
        \le O(\epsilon) \;
            \VnE{r}{t}{i} \VnE{s}{t}{j} \;
            \norm{\At} \norm{\Bt}
    \,.
\end{equation*}
\end{proposition}
\begin{proof}
If the last step is an \Ostep{}, then following the proof of
\cref{prop:out_scaling_bound}, $\| \M{A}^{(t)} \| = \| \M{B}^{(t)} \| = 1$ and
$\abs{c'_{ij} - \hat{c}'_{ij}} \le O(\epsilon)$.  If the last step is an
\Istep{}, then by \cref{prop:in_scaling_bound},
$$\abs{c'_{ij} - \hat{c}'_{ij}} \le O(\epsilon) \max_{i, k, j} \abs{a^{(t)}_{ik}}\abs{b^{(t)}_{kj}} \le O(\epsilon) \| \M{A}^{(t)} \| \| \M{B}^{(t)} \|.$$
The result then follows from the fact that $\M{C} - \hat{\M{C}} = \leftscalemat (\M{C}' - \hat{\M{C}}') \rightscalemat$.
\end{proof}

We now state the main result of this section.

\begin{theorem}\label{thm:convergence-alt-scaling}
The sequence
\begin{equation*}
    \VnE{r}{t}{i} \VnE{s}{t}{j} \norm{\At} \norm{\Bt} \qquad
    \text{for $t = 0, 1, \dotsc$}
\end{equation*}
is monotonically nonincreasing and converges linearly.
\end{theorem}
\begin{proof}
See \cref{app:convergence-alt-scaling}.
\end{proof}

This result implies that we can safely iteratively apply inside and outside scaling to improve the error bounds, but this process provides diminishing returns.

Next, we introduce our stopping criterion, which requires the following additional notation. 
Whenever step $t$ is an \Ostep{}, we use $\rpti{t}{i}$ and $\sptj{t}{j}$ to denote the diagonal elements of the matrices $\leftscalemat'$ and $\rightscalemat'$, respectively, that we compute in step $t$. Similarly, if step $t$ is an \Istep{}, $\pptk{t}{k}$ denotes the diagonal elements of~$\inscalemat$.

The stopping criterion works as follows. We test the intermediate
scaling factors $\pptk{t}{k}$, $\rpti{t}{i}$ and $\sptj{t}{j}$ in each iteration
starting with the one that immediately follows the
first \Ostep{}. In the \Isteps{}, we test whether all of the $\pptk{t}{k}$ fall
within the interval $\bigl[(1+\tau)^{-\frac{1}{4}},
  (1+\tau)^{\frac{1}{4}}\bigr]$, and in the \Osteps{}, we test whether all of
the $\rpti{t}{i}$ and $\sptj{t}{j}$ are greater than the threshold
$(1+\tau)^{-1/2}$. Whenever one of these conditions is true,
\cref{thm:stopping-criterion} below states that we are within a relative distance
$\tau$ from the limit, and so we stop iterating.  In practice, we may just
specify $t$ steps of scaling a priori, so as to have a better handle on
the overhead of the pre-processing.  We explore the performance overhead of the
pre-processing in \cref{sec:scaling_experiments}.

We now state the theorem that justifies the stopping criterion.
As we show in \cref{app:convergence-alt-scaling}, the sequences $\VnE{r}{t}{i}$, $\VnE{s}{t}{j}$, $\norm{\At}$ and $\norm{\Bt}$ converge. We use a superscript $\star$ to denote their limits, so that
\begin{equation*}
    \VnE{r}{t}{i} \to \VnE{r}{\star}{i} \,, \qquad
        \VnE{s}{t}{j} \to \VnE{s}{\star}{j} \,, \qquad
    \norm{\At} \to \norm{\At[\star]} \,, \qquad
        \norm{\Bt} \to \norm{\Bt[\star]} \,,
\end{equation*}
and we let
\begin{equation*}
    \VnE{\mu}{t}{ij} = 
    \frac{\abs[\big]{
        \VnE{r}{t}{i} \VnE{s}{t}{j} \norm{\At} \norm{\Bt}
        -
        \VnE{r}{\star}{i} \VnE{s}{\star}{j} \norm{\At[\star]} \norm{\Bt[\star]}
    }}{\abs[\big]{
        \VnE{r}{\star}{i} \VnE{s}{\star}{j} \norm{\At[\star]} \norm{\Bt[\star]}
    }}
\end{equation*}
be the relative distance of $\VnE{r}{t}{i} \VnE{s}{t}{j} \norm{\At} \norm{\Bt}$ from the limit.
We use $t_0$ to denote the index of the first \Ostep{} of the iteration.

\begin{theorem}\label{thm:stopping-criterion}
Let $\tau > 0$ be a user-specified tolerance parameter. We have that for $t = t_0, t_0+2, \dotsc$,
\begin{equation*}
        \max_{i,j} \VnE{\mu}{t+1}{ij} \leq \tau
        \quad \text{ if } \quad
        \min_k \pptk{t+1}{k} \geq (1 + \tau)^{-\frac{1}{4}}
        \quad \text{ and } \quad  
        \max_k\pptk{t+1}{k} \leq (1 + \tau)^\frac{1}{4},
\end{equation*}
and
\begin{equation*}
        \max_{i,j} \VnE{\mu}{t+2}{ij} \leq \tau 
        \quad \text{ if } \quad
        \min_i \rpti{t+2}{i} \geq (1 + \tau)^{-\frac{1}{2}}
        \quad \text{ and } \quad  
        \min_j \sptj{t+2}{j} \geq (1 + \tau)^{-\frac{1}{2}}.
\end{equation*}
\end{theorem}
\begin{proof}
See \cref{app:stopping-criterion}.
\end{proof}

Finally, we note that \cref{alg:alt_scaling} does not specify which form of
scaling to apply first.  While the analysis in this section applies to either
choice, we note that the limits of the two sequences are not identical---they
depend on which step is applied first.  We conclude this subsection with
examples demonstrating that these two choices can produce significantly
different results (in the case of one iteration of \cref{alg:alt_scaling}) and
that neither choice is always preferable.  \cref{ex:inout_outin_scaling} shows a
case where performing outside followed by inside scaling is more accurate than
performing inside followed by outside scaling; the opposite is true for the case
of \cref{ex:outin_inout_scaling}.

\begin{example}\label{ex:inout_outin_scaling}
Consider the matrices
\begin{equation*}
\M{A} =
\begin{bmatrix}
1 & \smallentry^{-1} \\
1 & 1 \\
\end{bmatrix},\quad
\M{B} =
\begin{bmatrix}
\smallentry & 1 \\
\smallentry & 1  \\
\end{bmatrix},\quad
\M{C} = \M{A}\cdot\M{B} =
\begin{bmatrix}
1 + \smallentry & 1 + \smallentry^{-1} \\
2\smallentry & 2  \\
\end{bmatrix}
\end{equation*}
for small $\smallentry > 0$.
We consider one step of alternating scaling.
Performing outside and then inside scaling computes
\begin{equation}\label{eqn:out_in_mats}
\M{A}' \leftarrow
\begin{bmatrix}
\smallentry & 1 \\
1 & 1 \\
\end{bmatrix},\quad
\M{B}' \leftarrow
\begin{bmatrix}
1 & 1 \\
1 & 1 \\
\end{bmatrix},\quad
\M{C}' \leftarrow
\begin{bmatrix}
1 + \smallentry & 1 + \smallentry  \\
2 & 2 \\
\end{bmatrix}
\end{equation}
and inside followed by outside scaling computes
\begin{equation}\label{eqn:in_out_mats}
\M{A}' \leftarrow
\begin{bmatrix}
\smallentry^{1/2} & 1 \\
1 & \smallentry^{1/2} \\
\end{bmatrix},\quad
\M{B}' \leftarrow
\begin{bmatrix}
\smallentry^{1/2} & \smallentry^{1/2} \\
1 & 1 \\
\end{bmatrix},\quad
\M{C}' \leftarrow
\begin{bmatrix}
1 + \smallentry^{1/4} & 1 + \smallentry^{1/4} \\
2 \cdot \smallentry^{1/2} & 2 \cdot \smallentry^{1/2} \\
\end{bmatrix}
\end{equation}

Consider the computation of entry $\ME{C}{21}$ with Strassen's algorithm:
\begin{align*}
m_2 = (\ME{A}{21}' + \ME{A}{22}')\ME{B}{11}',\quad
m_4 = \ME{A}{22}'(\ME{B}{21}' - \ME{B}{11}'),\quad
\ME{C}{21}' = m_2 + m_4.
\end{align*}
With $\M{A}'$ and $\M{B}'$ in \cref{eqn:out_in_mats}, all sub-terms are $O(1)$ and $\ME{C}{21}'$ is $O(1)$,
whereas for $\M{A}'$ and $\M{B}'$ in \cref{eqn:in_out_mats}, there are $O(\smallentry^{1/4})$ sub-terms and $\ME{C}{21}'$ is $O(\smallentry^{1/2})$.

From \cref{prop:alt-scaling}, the absolute error bound for entry $\ME{C}{21}=O(z)$ with no scaling is $O(1/z)$, with outside and then inside is $O(z)$, and with inside and then outside is $O(1/z^{1/2})$.
Thus, using only one step of \cref{alg:alt_scaling}, the accuracy of starting with an \Ostep{} can be much better than that of starting with an \Istep{}.
\end{example}

\begin{example}\label{ex:outin_inout_scaling}
Consider the matrices
\begin{equation*}
\M{A} =
\begin{bmatrix}
1 & \smallentry \\
\smallentry & \smallentry  \\
\end{bmatrix},\quad
\M{B} =
\begin{bmatrix}
\smallentry & 1 \\
1 & \smallentry^{-1}  \\
\end{bmatrix},\quad
\M{C} = \M{A}\cdot\M{B} =
\begin{bmatrix}
2\smallentry & 2 \\
\smallentry+\smallentry^2 & 1+\smallentry  \\
\end{bmatrix}
\end{equation*}
for small $\smallentry > 0$.
We consider one step of alternating scaling.
Performing outside and then inside scaling computes
\begin{equation}\label{eqn:out_in_mats2}
\M{A}' \leftarrow
\begin{bmatrix}
1 & \smallentry \\
1 & 1 \\
\end{bmatrix},\quad
\M{B}' \leftarrow
\begin{bmatrix}
\smallentry & 1 \\
\smallentry & 1 \\
\end{bmatrix},\quad
\M{C}' \leftarrow
\begin{bmatrix}
\smallentry + \smallentry^2 & 1 + \smallentry  \\
2\smallentry & 2 \\
\end{bmatrix}
\end{equation}
and inside followed by outside scaling computes
\begin{equation}\label{eqn:in_out_mats2}
\M{A}' \leftarrow
\begin{bmatrix}
1 & 1 \\
\smallentry & 1 \\
\end{bmatrix},\quad
\M{B}' \leftarrow
\begin{bmatrix}
1 & 1 \\
1 & 1 \\
\end{bmatrix},\quad
\M{C}' \leftarrow
\begin{bmatrix}
2 & 2 \\
1+ \smallentry & 1 + \smallentry \\
\end{bmatrix}.
\end{equation}

Consider the computation of entry $\ME{C}{21}$ with Strassen's algorithm:
\begin{align*}
m_2 = (\ME{A}{21}' + \ME{A}{22}')\ME{B}{11}',\quad
m_4 = \ME{A}{22}'(\ME{B}{21}' - \ME{B}{11}'),\quad
\ME{C}{21}' = m_2 + m_4.
\end{align*}
With $\M{A}'$ and $\M{B}'$ in \cref{eqn:out_in_mats2}, $\ME{C}{21}'$ is $O(\smallentry)$ but there are $O(1)$ subterms,
whereas for $\M{A}'$ and $\M{B}'$ in \cref{eqn:in_out_mats2}, $\ME{C}{21}'$ and all subterms are $O(1)$.
This case illustrates that using only one step of \cref{alg:alt_scaling}, the accuracy of starting with an \Istep{} can be much better than that of starting with an \Ostep{}.
\end{example}

\subsection{Scaling is not always enough}

We now provide a simple example that shows how Strassen's algorithm computes a result with large relative error, using any of the scaling algorithms presented in this section.

\begin{example}\label{ex:scaling_fails}
Consider the matrices
\begin{equation}\label{eqn:scaling_fails}
\M{A} =
\begin{bmatrix}
1 & \smallentry \\
\smallentry & 1 \\
\end{bmatrix},\;
\M{B} =
\begin{bmatrix}
1 & \smallentry \\
\smallentry & 1 \\
\end{bmatrix},\;
\M{C} = \M{A}\cdot\M{B} =
\begin{bmatrix}
1 + \smallentry^2 & 2\smallentry \\
2\smallentry & 1 + \smallentry^2  \\
\end{bmatrix}
\end{equation}
for small $\smallentry > 0$.
In this case, both outside and inside scaling leave the matrix unchanged.
When computing $\ME{C}{12}$,
\begin{align*}
m_3             &= \ME{A}{11}(\ME{B}{12} - \ME{B}{22}) = 1(z - 1) \\
m_5             &= (\ME{A}{11} + \ME{A}{12})\ME{B}{22} = (1 + z)1 \\
\ME{C}{12} &= m_3 + m_5,
\end{align*}
There are subterms on the order of unity, so the relative error is $O(1 / \smallentry)$.
\end{example}

\subsection{Numerical experiments}\label{sec:scaling_experiments}

We tested the scaling algorithms on samples of random matrices whose entries
were not as contrived as those in the prior sections.  We used a sample of
$\M{A} \in \mathbb{R}^{N \times N}$ and $\M{B}^{N \times N}$ from the following
distributions:
\begin{enumerate}
\item $\ME{A}{ij}, \ME{B}{ij} \sim$ Uniform(0, 1)
\item $\ME{A}{ij} \sim \text{Uniform}(0, 1 / N^2)$ if $j > N / 2$, otherwise, $\ME{A}{ij} \sim \text{Uniform}(0, 1)$;
$\ME{B}{ij} \sim \text{Uniform}(0, 1 / N^2)$ if $i < N / 2$, otherwise $\ME{B}{ij} \sim \text{Uniform}(0, 1)$
\item $\ME{A}{ij} \sim \text{Uniform}(0, N^2)$ if $i < N / 2$ and $j > N / 2$, otherwise, $\ME{A}{ij} \sim \text{Uniform}(0, 1)$;
$\ME{B}{ij} \sim \text{Uniform}(0, 1 / N^2)$ if $j < N / 2$, otherwise $\ME{B}{ij} \sim \text{Uniform}(0, 1)$
\end{enumerate}
Samples from the first distribution are well-behaved for fast matrix
multiplication algorithms.  On the other hand, samples from the second and third
distributions are adversarial and model the matrices in
\cref{ex:in_scaling,ex:inout_outin_scaling}, respectively.

We sampled 100 pairs of matrices ($N = 2000$) from each distribution and
computed the error of Strassen's algorithm with $L$ recursive levels, $L = 1, 2,
\ldots, 6$.  Specifically, the error was the maximum value of $\max_{ij}
\vert\hat{c}_{ij} - \ME{C}{ij}\vert / \vert \ME{C}{ij} \vert$ over the 100
samples, where $\M{C}$ was computed with quadruple precision.
\cref{fig:scaling} plots these errors.  For the first probability distribution,
the relative errors are all roughly the same.  With the second distribution,
only inside-outside scaling or 2-times repeated outside-inside scaling compute
relatively accurate solutions.  In this case, inside and outside-inside scaling
are moderately more accurate than no scaling or outside scaling, but they still
produce relative errors several orders of magnitude larger than the best case.
Finally, for the third distribution, inside scaling and no scaling result in
much larger relative errors, and inside-outside scaling is slightly worse than
outside, outside-inside, or 10-times repeated outside-inside scaling.  These
experiments demonstrate that with no prior knowledge of the distribution,
repeated outside-inside scaling is the safe choice for fast matrix
multiplication.

\begin{figure}
\centering
\includegraphics[width=4.2cm]{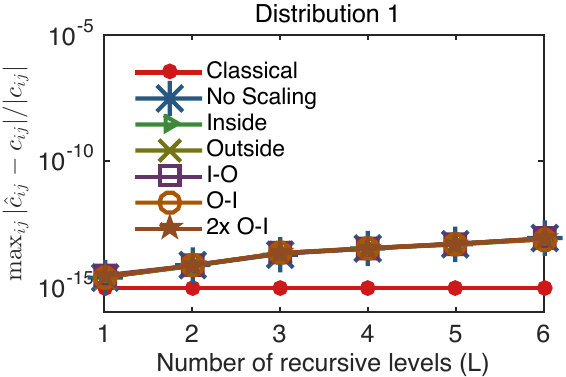}
\includegraphics[width=4.2cm]{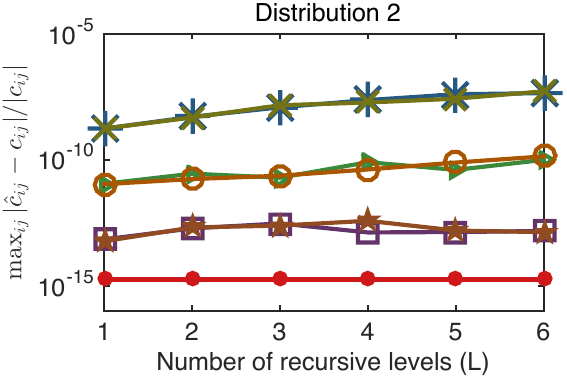}
\includegraphics[width=4.2cm]{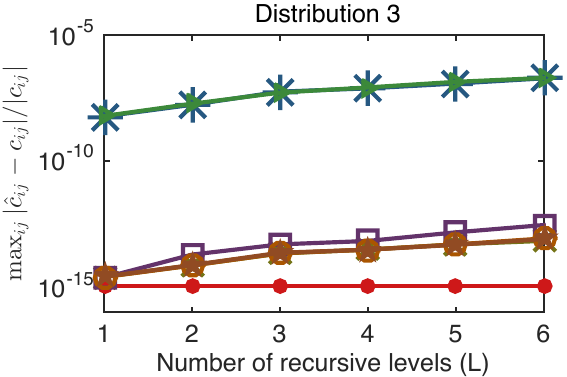}
\caption{%
Relative error of Strassen's algorithm as a function of the number of recursive
steps, $L$, for several scaling techniques.  The results in each plot are for
matrices $\M{A}$ and $\M{B}$ sampled from different probability distributions.
(Top) Stability is well-behaved, and no scaling is necessary for small relative
errors.  (Middle) The matrices are adversarial and inside-outside or 2-times
repeated outside-inside scaling have the smallest relative errors.  (Right) The
matrices are adversarial and outside, outside-inside, or 2-times repeated
outside-inside scaling have the smallest relative errors.
}
\label{fig:scaling}
\end{figure}

Each iteration of outside or inside scaling is $O(MK+KN+MN)$ flops, so scaling
does not affect the asymptotic performance.  However, quadratic costs do affect
the practical implementation of fast matrix
multiplication~\cite{benson2014framework}.  Subsequently we tested the
performance impact of scaling.  We use \emph{effective
  gflops}~\cite{benson2014framework,lipshitz2012communication} to measure the
performance of multiplying an $M \times K$ matrix by a $K \times N$ matrix:
\begin{equation}
\frac{2 \cdot MKN - MN}{\text{time in seconds}} \cdot \text{1e-9}.
\end{equation}
This lets us compare fast matrix multiplication algorithms to the classical
algorithm on a familiar inverse-time scale.  All experiments were conducted on a
single compute node on NERSC's Edison machine.  Each node has two 12-core Intel
2.4 GHz Ivy Bridge processors and 64 GB of memory.  Our experiments were
single-threaded.  We report the median of five trials for each timing result.

\cref{fig:scaling_performance} summarizes the performance results for Strassen's
algorithm ($L = 1$), with and without two steps of \cref{alg:alt_scaling}, for
multiplying square matrices of dimension $N$.  There is a noticeable impact on
performance.  Strassen's without scaling out-performs the classical algorithm
for $N \ge 2500$ while scaling pushes this threshold to $N \ge 3500$.  As $N$
grows, the performance impact of scaling gets smaller.  This follows from the
asymptotic analysis---as $N$ grows, the impact from quadratic terms shrinks.

\begin{figure}
\centering
\includegraphics[width=6.5cm]{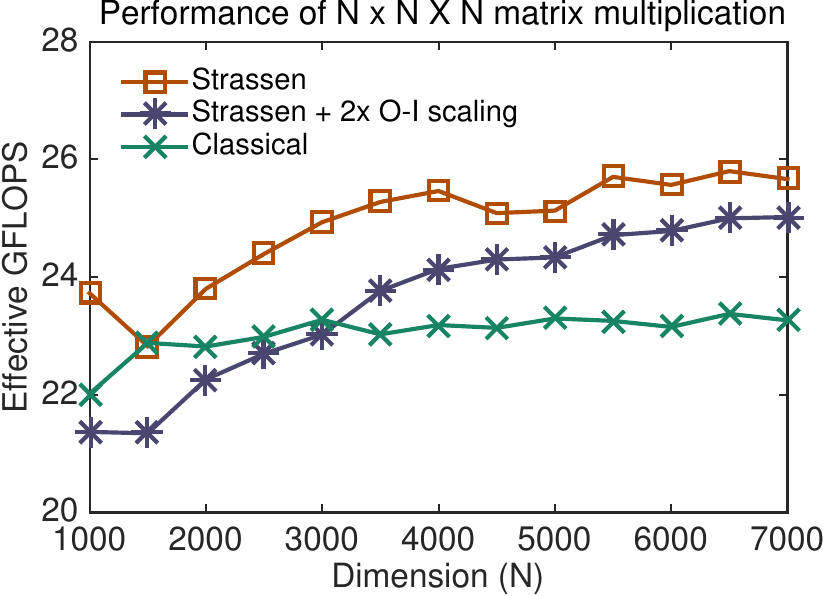}
\caption{%
Performance of Strassen's algorithm ($L = 1$), with and without two steps of
inside-outside scaling, and the classical algorithm.
}
\label{fig:scaling_performance}
\end{figure}

\section{Discussion}
\label{sec:discussion}

One of the central components of our algorithmic error analysis is that two
data-independent quantities drive the error bounds for fast matrix
multiplication.  First, $\qmax$ captures the accumulation error from adding
matrices.  Second, $\emax$ bounds the growth in the magnitude of intermediate
terms.  Our results in \cref{sec:selection} show that having a small $\emax$ is
important, but this does not fully characterize stability in practice.  The same
result has been observed when comparing Strassen's algorithm and the Winograd
variant~\cite{higham2002accuracy}.  An encouraging result from our experiments
is that the number of non-zeroes in the $\M U$, $\M V$, and $\M W$ matrices,
which determines the constant in the computational complexity, is positively
correlated with $\emax$.  In other words, for a given base case and algorithm
rank, by improving performance, we generally also improve stability.  Another
lesson from our analysis is that we should not think of using fast algorithms
asymptotically but rather as having a fixed number of recursive levels.  This
leads to better performance in practice~\cite{benson2014framework} and also to
the improved error bounds and numerical stability presented in
\cref{sec:bounds,sec:selection}.  Finally, because the principal quantities for
understanding algorithmic error ($\emax$ and $\qmax$) are independent of the
asymptotic complexity, we have new metrics over which to optimize when searching
for fast matrix multiplication algorithms.

For performance reasons, the best choice of fast algorithm depends on the shape
of the matrices being multiplied~\cite{benson2014framework}.  In general, a
choice of algorithm can be made at each recursive level.  Subsequently, we
believe that uniform non-stationary algorithms are the right choice in practice
for achieving the best performance.  \Cref{thm:non_stationary_analysis} provides
the appropriate error bounds for this case.

The analysis in \cref{sec:non_uniform_non_stationary} formalizes the error
analysis for existing techniques to improve stability of Strassen's algorithm
and the Winograd variant~\cite{castrapel2007precision,d2014better} and also
generalizes the approach for all fast matrix multiplication algorithms.  The
analysis provides the formula over which to optimize when considering
non-uniform, non-stationary algorithms.  However, finding the best algorithm is
a combinatorial optimization problem that grows exponentially in the number of
recursive levels.  Algorithm design in this space is an interesting avenue for
future research.

Using these algorithmic techniques improves the normwise accuracy of the
computed product.  However, because the errors are normwise, small elements of
the product can be computed less accurately than warranted by their condition
numbers.  By pre- and post-processing the data, we can improve componentwise
accuracy as well.  Specifically, we analyzed a hierarchy of diagonal scaling
techniques that reduce the number of cases where fast matrix multiplication
yields relatively inaccurate small entries in the product.  Nevertheless, there
are cases that cannot be solved by our diagonal scaling algorithms (\eg,
\cref{ex:scaling_fails}).  When scaling helps, a couple of iterations are
sufficient, and this is backed up by \cref{thm:convergence-alt-scaling}.

The asymptotic operation cost of diagonal scaling is proportional to the size of the matrices, so it is
dominated by cost of current matrix multiplication algorithms.  In our
experiments, we found that scaling does incur a noticeable performance penalty
for reasonably sized matrices, but fast matrix multiplication with diagonal
scaling still can outperform the classical algorithm.  We note that our diagonal
scaling implementation is not fully optimized; for example, it is possible to
overlap inside and outside scaling and delay updating actual matrix entries,
which can both reduce the memory traffic overhead.

\section*{Acknowledgments}

We thank Madan Musuvathi for providing the $\bc{2}{3}{2}$ algorithms from the SAT solver.
We also thank the editor and anonymous reviewers for their suggestions, which make the paper better and more readable.

\bibliographystyle{siamplain}
\bibliography{refs}

\appendix

\section{Strassen's Algorithm}
\label{app:strassen}

Strassen's algorithm~\cite{strassen1969gaussian} is a $\bc{2}{2}{2}$ algorithm
specified by the following $\M{U}$, $\M{V}$, and $\M{W}$ matrices:
\newcommand{\ph}{\phantom{-}}
\begin{align*}
\M{U} &=
\begin{bmatrix}
\ph1 & \ph0 & \ph1 & \ph0 & \ph1 & -1 & \ph0 \ph \\
\ph0 & \ph0 & \ph0 & \ph0 & \ph1 &  \ph0 & \ph1 \ph \\
\ph0 & \ph1 & \ph0 & \ph0 & \ph0 &  \ph1 & \ph0 \ph \\
\ph1 & \ph1 & \ph0 & \ph1 & \ph0 &  \ph0 & -1 \ph \\
\end{bmatrix} \\
\M{V} &= 
\begin{bmatrix}
\ph1 & \ph1 & \ph0  & -1 & \ph0 & \ph1 & \ph0 \ph \\
\ph0 & \ph0 & \ph1  & \ph0  & \ph0 & \ph1 & \ph0 \ph \\
\ph0 & \ph0 & \ph0  & \ph1  & \ph0 & \ph0 & \ph1 \ph \\
\ph1 & \ph0 & -1 & \ph0  & \ph1 & \ph0 & \ph1 \ph \\
\end{bmatrix} \\
\M{W} &= 
\begin{bmatrix}
\ph1 &  \ph0 & \ph0 & \ph1 & -1 & \ph0 & \ph1 \ph \\
\ph0 &  \ph1 & \ph0 & \ph1 &  \ph0 & \ph0 & \ph0 \ph \\
\ph0 &  \ph0 & \ph1 & \ph0 &  \ph1 & \ph0 & \ph0 \ph \\
\ph1 & -1 & \ph1 & \ph0 &  \ph0 & \ph1 & \ph0 \ph \\
\end{bmatrix}. 
\end{align*}
Note that the rows of $\M U$ and $\M V$ correspond to a column-major ordering of
the entries of the input matrices and the rows of $\M W$ correspond to a
row-major ordering of the output matrix, following the convention of previous
work \cite{Brent70,JM86}.  We point out that this algorithm is cyclic-invariant,
so that $\alg{\M U}{\M V}{\M W}=\alg{\M W}{\M U}{\M V}=\alg{\M V}{\M W}{\M U}$
(up to permutations on the columns of the matrices), which implies that all
three rotations have the same $\qmax$ and $\emax$ values.

\section{$\bc{3}{2}{3}$ fast matrix multiplication algorithm}
\label{app:fast323}

The following algorithm for base case $\bc{3}{2}{3}$ has 94 nonzeros with $\emax=20$ and $\qmax=10$.
\setcounter{MaxMatrixCols}{15}
\newcommand{\scaleboxsize}{0.95}
\begin{align*}
\label{eqn:fast323}
\M{U} &=
\scalebox{\scaleboxsize}{$
\begin{bmatrix}
\ph0 & \ph1 & \ph0 & \ph0 & -1 & \ph1 & \ph1 & \ph0 & \ph0 & \ph0 & -1 & \ph0 & \ph0 & \ph0 & -1 \ph \\
\ph0 & \ph1 & \ph0 & \ph0 & \ph0 & \ph0 & \ph0 & -1 & \ph0 & -1 & \ph0 & \ph0 & \ph1 & -1 & \ph0 \ph \\ 
\ph0 & \ph0 & \ph0 & \ph1 & \ph0 & \ph0 & \ph0 & -1 & \ph0 & -1 & \ph1 & \ph0 & \ph0 & -1 & \ph0 \ph \\ 
-1 & \ph0 & \ph1 & \ph0 & \ph1 & \ph0 & \ph0 & \ph0 & \ph0 & \ph0 & \ph1 & \ph0 & \ph0 & \ph0 & \ph1 \ph \\
-1 & \ph0 & \ph0 & \ph0 & \ph0 & \ph1 & \ph0 & \ph1 & \ph0 & \ph1 & \ph0 & \ph1 & -1 & \ph0 & \ph0 \ph \\
\ph0 & \ph0 & -1 & \ph0 & \ph0 & \ph0 & \ph0 & \ph1 & \ph1 & \ph0 & -1 & \ph0 & \ph0 & \ph0 & -1 \ph \\
\end{bmatrix}
$}  \\
\M{V} &= 
\scalebox{\scaleboxsize}{$
\begin{bmatrix}
\ph0 & \ph0 & \ph1 & \ph1 & \ph0 & \ph0 & \ph0 & \ph0 & \ph0 & \ph0 & \ph1 & \ph0 & \ph0 & \ph1 & \ph1 \ph \\
\ph0 & \ph0 & \ph1 & \ph0 & \ph0 & \ph0 & \ph0 & -1 & \ph1 & -1 & \ph0 & \ph0 & \ph0 & \ph1 & \ph0 \ph \\
\ph0 & -1 & \ph0 & \ph0 & \ph0 & \ph1 & \ph0 & \ph0 & \ph0 & \ph1 & \ph0 & \ph1 & \ph1 & -1 & \ph0 \ph \\
-1 & \ph0 & \ph0 & \ph0 & \ph0 & \ph0 & \ph0 & \ph1 & \ph0 & \ph1 & \ph0 & \ph1 & \ph0 & -1 & \ph0 \ph \\
\ph0 & \ph1 & \ph1 & \ph0 & \ph0 & \ph0 & \ph1 & \ph0 & \ph0 & \ph0 & \ph1 & \ph0 & \ph0 & \ph0 & \ph1 \ph \\
\ph1 & \ph0 & \ph1 & \ph0 & \ph1 & \ph1 & \ph1 & \ph0 & \ph0 & \ph0 & \ph0 & \ph0 & \ph0 & \ph0 & \ph1 \ph \\
\end{bmatrix} 
$} \\
\M{W} &= 
\scalebox{\scaleboxsize}{$
\begin{bmatrix}
\ph0 & \ph0 & \ph1 & \ph0 & \ph0 & \ph0 & -1 & \ph0 & \ph1 & \ph0 & \ph0 & \ph0 & \ph0 & \ph0 & -1 \ph \\
\ph1 & \ph0 & \ph0 & \ph0 & \ph1 & \ph1 & \ph0 & \ph0 & \ph0 & \ph0 & \ph0 & -1 & \ph0 & \ph0 & \ph0 \ph \\
\ph0 & \ph0 & \ph0 & \ph0 & \ph1 & \ph0 & \ph1 & \ph0 & \ph0 & \ph0 & \ph0 & \ph0 & \ph0 & \ph0 & \ph0 \ph \\
\ph0 & \ph0 & \ph0 & -1 & \ph0 & \ph0 & \ph0 & \ph0 & \ph0 & -1 & \ph0 & \ph1 & \ph0 & -1 & \ph0 \ph \\
\ph0 & \ph0 & \ph0 & \ph0 & \ph0 & \ph0 & \ph0 & \ph0 & \ph0 & \ph0 & \ph0 & \ph1 & \ph1 & \ph0 & \ph0 \ph \\
\ph0 & \ph1 & \ph0 & \ph0 & \ph0 & \ph1 & -1 & \ph0 & \ph0 & \ph0 & \ph0 & \ph0 & \ph1 & \ph0 & \ph0 \ph \\
\ph0 & \ph0 & \ph0 & \ph1 & \ph0 & \ph0 & \ph0 & \ph0 & \ph1 & \ph0 & \ph0 & \ph0 & \ph0 & \ph0 & \ph0 \ph \\
\ph0 & \ph0 & \ph0 & \ph0 & \ph0 & \ph0 & \ph0 & \ph1 & \ph1 & -1 & \ph0 & \ph0 & -1 & \ph0 & \ph0 \ph \\
\ph0 & \ph0 & \ph0 & -1 & \ph1 & \ph0 & \ph0 & \ph0 & \ph0 & \ph0 & \ph1 & \ph0 & \ph0 & \ph0 & -1 \ph \\
\end{bmatrix}
$}. 
\end{align*}
Note that the rows of $\M U$ and $\M V$ correspond to a column-major ordering of the entries of the input matrices and the rows of $\M W$ correspond to a row-major ordering of the output matrix, which implies that $\alg{\M W}{\M U}{\M V}$ is an algorithm for $\bc{3}{3}{2}$ and $\alg{\M V}{\M W}{\M U}$ is an algorithm for $\bc{2}{3}{3}$.

\section{$\bc{4}{4}{2}$ fast matrix multiplication algorithm}
\label{app:fast442}

The following algorithm specifies a rank $26$ fast matrix multiplication
algorithm with base case $\bc{4}{4}{2}$.

\setcounter{MaxMatrixCols}{26}
\renewcommand{\scaleboxsize}{0.55}
\begin{align*}
\label{eqn:fast442}
\M{U} &= \frac12
\scalebox{\scaleboxsize}{$
\begin{bmatrix}
\ph 2 & \ph 0 & \ph 0 & -2 & \ph 0 & \ph 0 & \ph 0 & \ph 0 & -2 & \ph 0 & \ph 0 & \ph 0 & \ph 0 & -2 & \ph 0 & \ph 0 & \ph 0 & \ph 0 & \ph 0 & \ph 0 & \ph 0 & \ph 0 & \ph 0 & \ph 0 & \ph 0 & \ph 0 \\
\ph 0 & \ph 0 & \ph 0 & \ph 0 & \ph 0 & -2 & \ph 0 & \ph 0 & \ph 0 & \ph 0 & \ph 2 & \ph 2 & \ph 2 & \ph 0 & \ph 0 & \ph 0 & \ph 0 & \ph 0 & \ph 0 & \ph 0 & \ph 0 & \ph 0 & \ph 0 & \ph 0 & \ph 0 & \ph 0 \\
\ph 0 & \ph 0 & \ph 2 & \ph 0 & \ph 0 & -2 & \ph 0 & \ph 0 & \ph 0 & \ph 0 & \ph 2 & \ph 2 & \ph 2 & -2 & \ph 0 & \ph 0 & \ph 0 & \ph 0 & \ph 0 & \ph 0 & -2 & \ph 0 & \ph 0 & \ph 2 & \ph 0 & \ph 0 \\
\ph 0 & \ph 0 & \ph 0 & \ph 0 & \ph 0 & \ph 0 & \ph 0 & \ph 0 & \ph 0 & \ph 0 & \ph 0 & -2 & \ph 0 & \ph 2 & \ph 0 & \ph 0 & \ph 0 & \ph 0 & \ph 0 & \ph 0 & \ph 0 & \ph 0 & \ph 0 & \ph 0 & \ph 0 & \ph 0 \\
\ph 0 & \ph 2 & \ph 0 & \ph 0 & \ph 0 & \ph 0 & \ph 0 & \ph 1 & \ph 0 & \ph 0 & \ph 0 & \ph 0 & \ph 0 & \ph 0 & \ph 0 & \ph 0 & \ph 0 & \ph 0 & -1 & \ph 0 & \ph 0 & \ph 0 & \ph 0 & \ph 0 & \ph 0 & \ph 2 \\
\ph 0 & \ph 2 & \ph 0 & \ph 0 & \ph 0 & \ph 0 & \ph 0 & \ph 0 & \ph 0 & \ph 0 & \ph 0 & \ph 0 & -2 & \ph 0 & \ph 0 & \ph 1 & \ph 0 & \ph 0 & \ph 0 & \ph 1 & \ph 0 & -2 & \ph 0 & \ph 2 & -2 & \ph 0 \\
\ph 0 & \ph 2 & \ph 0 & \ph 0 & \ph 0 & \ph 0 & \ph 0 & \ph 1 & \ph 0 & \ph 0 & \ph 0 & \ph 0 & \ph 0 & \ph 0 & \ph 0 & \ph 0 & \ph 0 & \ph 0 & -1 & \ph 0 & \ph 0 & -2 & \ph 0 & \ph 0 & \ph 0 & \ph 0 \\
\ph 0 & \ph 0 & \ph 0 & \ph 1 & \ph 2 & \ph 0 & -1 & -1 & \ph 0 & -1 & -2 & \ph 2 & \ph 0 & -2 & -2 & \ph 0 & \ph 0 & \ph 1 & \ph 0 & \ph 0 & \ph 1 & \ph 0 & \ph 0 & \ph 0 & \ph 0 & \ph 0 \\
\ph 0 & \ph 0 & \ph 0 & \ph 0 & \ph 0 & \ph 0 & \ph 0 & \ph 1 & \ph 0 & \ph 0 & \ph 0 & \ph 0 & \ph 0 & \ph 0 & \ph 0 & -1 & \ph 1 & \ph 1 & \ph 0 & \ph 0 & \ph 0 & \ph 0 & \ph 0 & \ph 0 & \ph 0 & \ph 0 \\
\ph 0 & \ph 2 & \ph 0 & \ph 0 & -2 & \ph 0 & \ph 0 & \ph 0 & \ph 0 & \ph 0 & \ph 0 & \ph 0 & \ph 0 & \ph 0 & \ph 0 & \ph 1 & \ph 0 & \ph 0 & \ph 0 & \ph 1 & \ph 0 & -2 & \ph 2 & \ph 0 & \ph 2 & \ph 0 \\
\ph 0 & \ph 2 & \ph 0 & \ph 0 & -2 & \ph 0 & \ph 0 & \ph 1 & \ph 0 & \ph 0 & \ph 0 & \ph 0 & \ph 0 & \ph 0 & \ph 0 & \ph 0 & \ph 0 & \ph 1 & \ph 0 & \ph 0 & \ph 0 & -2 & \ph 2 & \ph 0 & \ph 2 & \ph 0 \\
\ph 0 & \ph 0 & \ph 0 & \ph 1 & \ph 0 & \ph 0 & -1 & -1 & \ph 0 & \ph 1 & \ph 0 & \ph 0 & \ph 0 & \ph 0 & \ph 0 & \ph 0 & \ph 0 & -1 & \ph 0 & \ph 0 & -1 & \ph 0 & \ph 0 & \ph 0 & \ph 0 & \ph 0 \\
\ph 0 & -2 & \ph 0 & \ph 0 & \ph 0 & \ph 0 & \ph 0 & \ph 0 & \ph 2 & \ph 0 & \ph 0 & \ph 0 & \ph 0 & \ph 0 & \ph 0 & \ph 1 & -1 & -1 & \ph 1 & \ph 0 & \ph 0 & \ph 0 & \ph 0 & \ph 0 & \ph 0 & -2 \\
\ph 0 & \ph 0 & \ph 0 & \ph 0 & \ph 2 & \ph 0 & \ph 0 & \ph 0 & \ph 0 & \ph 0 & \ph 0 & \ph 0 & \ph 2 & \ph 0 & \ph 0 & \ph 0 & \ph 0 & \ph 0 & \ph 0 & \ph 0 & \ph 0 & \ph 0 & \ph 0 & \ph 0 & \ph 0 & \ph 0 \\
\ph 0 & \ph 0 & \ph 0 & \ph 0 & \ph 2 & \ph 0 & \ph 0 & \ph 0 & \ph 0 & \ph 0 & \ph 0 & \ph 0 & \ph 0 & \ph 0 & \ph 0 & \ph 0 & \ph 0 & -1 & \ph 1 & \ph 0 & \ph 0 & \ph 0 & \ph 0 & \ph 0 & -2 & \ph 0 \\
\ph 0 & \ph 0 & \ph 0 & \ph 0 & -2 & \ph 0 & \ph 0 & \ph 0 & \ph 0 & \ph 0 & \ph 2 & \ph 0 & \ph 0 & \ph 2 & \ph 2 & \ph 0 & \ph 0 & \ph 0 & \ph 0 & \ph 0 & \ph 0 & \ph 0 & \ph 0 & \ph 0 & \ph 0 & \ph 0 \\
\end{bmatrix}$
} \\
\M{V} &= \frac12
\scalebox{\scaleboxsize}{$
\begin{bmatrix}
\ph 2 & \ph 0 & -2 & \ph 0 & \ph 0 & -2 & \ph 0 & \ph 0 & \ph 0 & \ph 0 & -2 & \ph 0 & \ph 0 & -2 & -2 & \ph 0 & \ph 0 & \ph 0 & \ph 0 & \ph 0 & \ph 0 & \ph 0 & \ph 0 & \ph 0 & \ph 0 & \ph 0 \\
\ph 2 & \ph 1 & \ph 2 & \ph 2 & \ph 0 & \ph 0 & \ph 0 & \ph 2 & \ph 2 & \ph 0 & \ph 0 & \ph 0 & \ph 0 & \ph 0 & \ph 0 & \ph 0 & \ph 2 & -2 & -2 & \ph 2 & \ph 2 & \ph 0 & \ph 0 & -2 & \ph 1 & \ph 0 \\
-2 & \ph 0 & -2 & -2 & \ph 0 & -2 & -2 & \ph 0 & \ph 0 & -2 & \ph 0 & -2 & \ph 0 & \ph 0 & \ph 0 & \ph 0 & \ph 0 & \ph 0 & \ph 0 & \ph 0 & -2 & \ph 0 & \ph 0 & \ph 0 & \ph 0 & \ph 0 \\
-2 & \ph 0 & \ph 2 & -2 & \ph 0 & \ph 0 & -2 & \ph 0 & -2 & \ph 2 & \ph 0 & \ph 0 & \ph 0 & \ph 0 & \ph 0 & -2 & -2 & \ph 0 & \ph 0 & \ph 2 & \ph 2 & \ph 0 & -2 & -2 & \ph 0 & \ph 0 \\
\ph 2 & \ph 1 & -2 & \ph 2 & \ph 0 & \ph 0 & \ph 0 & \ph 2 & \ph 2 & \ph 0 & \ph 0 & \ph 0 & \ph 0 & \ph 0 & \ph 0 & \ph 2 & \ph 0 & \ph 2 & \ph 2 & \ph 0 & -2 & \ph 0 & \ph 0 & \ph 2 & -1 & \ph 2 \\
\ph 2 & \ph 1 & -2 & \ph 2 & \ph 0 & \ph 2 & \ph 0 & \ph 2 & \ph 2 & \ph 0 & \ph 2 & \ph 0 & -2 & \ph 0 & \ph 2 & \ph 2 & \ph 0 & \ph 2 & -2 & \ph 0 & -2 & \ph 0 & \ph 0 & \ph 2 & -1 & \ph 0 \\
-2 & -2 & -2 & -2 & \ph 0 & \ph 0 & -2 & \ph 0 & -2 & -2 & \ph 0 & \ph 0 & \ph 0 & \ph 0 & \ph 0 & \ph 0 & \ph 0 & \ph 0 & \ph 0 & \ph 0 & -2 & \ph 2 & \ph 2 & \ph 2 & \ph 0 & -2 \\
-2 & \ph 0 & -2 & -2 & -2 & \ph 2 & -2 & \ph 0 & -2 & -2 & \ph 2 & \ph 0 & -2 & \ph 0 & \ph 0 & \ph 0 & \ph 0 & \ph 0 & \ph 0 & \ph 0 & -2 & \ph 0 & -2 & \ph 2 & -2 & \ph 0 \\
\end{bmatrix} $
}  \\
\M{W} &= \phantom{\frac12}
\scalebox{\scaleboxsize}{$
\begin{bmatrix}
\ph 1 & -1 & \ph 0 & \ph 1 & \ph 0 & \ph 0 & \ph 1 & \ph 1 & \ph 0 & \ph 0 & \ph 0 & \ph 0 & \ph 0 & \ph 0 & \ph 0 & \ph 1 & \ph 1 & \ph 0 & \ph 0 & \ph 1 & \ph 0 & \ph 1 & \ph 0 & \ph 0 & \ph 0 & \ph 1 \\
\ph 0 & \ph 1 & -1 & \ph 0 & \ph 0 & \ph 0 & \ph 0 & \ph 0 & \ph 0 & -1 & \ph 0 & \ph 0 & \ph 0 & \ph 0 & \ph 0 & -1 & -1 & -1 & -1 & -1 & -1 & -1 & \ph 0 & \ph 0 & \ph 0 & -1 \\
\ph 0 & -1 & \ph 0 & \ph 1 & \ph 0 & \ph 0 & \ph 1 & \ph 1 & -1 & \ph 0 & \ph 0 & \ph 0 & \ph 0 & \ph 0 & \ph 0 & \ph 1 & -1 & \ph 0 & \ph 0 & \ph 1 & \ph 0 & \ph 1 & \ph 0 & \ph 0 & \ph 0 & \ph 0 \\
\ph 0 & \ph 0 & \ph 0 & \ph 0 & \ph 0 & \ph 0 & \ph 0 & \ph 0 & \ph 0 & \ph 1 & \ph 0 & \ph 0 & \ph 0 & \ph 0 & \ph 0 & \ph 1 & \ph 1 & \ph 1 & \ph 0 & -1 & \ph 1 & -1 & -1 & -1 & \ph 1 & \ph 0 \\
\ph 0 & \ph 0 & \ph 1 & \ph 0 & \ph 1 & \ph 0 & \ph 0 & \ph 0 & \ph 0 & \ph 1 & -1 & \ph 0 & -1 & \ph 0 & \ph 1 & \ph 1 & \ph 1 & \ph 1 & \ph 0 & \ph 1 & \ph 1 & \ph 0 & \ph 0 & \ph 0 & \ph 1 & \ph 0 \\
\ph 1 & \ph 0 & -1 & \ph 1 & \ph 0 & \ph 0 & \ph 1 & \ph 0 & \ph 0 & -1 & \ph 0 & \ph 0 & \ph 0 & -1 & \ph 1 & \ph 0 & \ph 0 & \ph 0 & \ph 0 & \ph 0 & -1 & \ph 0 & \ph 0 & \ph 0 & \ph 0 & \ph 0 \\
\ph 0 & \ph 0 & \ph 0 & \ph 0 & -1 & \ph 1 & \ph 0 & \ph 0 & \ph 0 & -1 & \ph 1 & \ph 0 & \ph 0 & \ph 0 & -1 & -1 & -1 & -1 & \ph 0 & \ph 1 & -1 & \ph 0 & \ph 0 & \ph 1 & -1 & \ph 0 \\
\ph 0 & \ph 0 & \ph 0 & \ph 0 & \ph 0 & \ph 1 & \ph 1 & \ph 0 & \ph 0 & \ph 1 & \ph 1 & \ph 1 & \ph 0 & \ph 0 & -1 & \ph 0 & \ph 0 & \ph 0 & \ph 0 & \ph 0 & \ph 0 & \ph 0 & \ph 0 & \ph 0 & \ph 0 & \ph 0 \\
\end{bmatrix}$
} .
\end{align*}

Note that the rows of $\M U$ and $\M V$ correspond to a column-major ordering of
the entries of the input matrices and the rows of $\M W$ correspond to a
row-major ordering of the output matrix, which implies that $\alg{\M W}{\M U}{\M
  V}$ is an algorithm for $\bc{4}{2}{4}$ and $\alg{\M V}{\M W}{\M U}$ is an
algorithm for $\bc{2}{4}{4}$.  However, $\alg{\M V}{\M W}{\M U}$ yields an
$\emax=102$, which is greater than $\alg{\M U}{\M V}{\M W}$'s 89.  The $\emax$
value can be maintained for base case $\bc{2}{4}{4}$ by using a different
transformation that corresponds to transposing the matrix multiplication:
$\alg{\VP{4}{2} \M V}{\VP{4}{4} \M U}{\VP{2}{4} \M W}$, where $\VP{m}{n}$ is the
so-called vec permutation matrix \cite{HS81}, exchanging column-ordering for
row-ordering for a vectorized $m\times n$ matrix.

\section{Convergence analysis of alternating scaling}
\label{app:convergence-alt-scaling}

In this appendix we prove \cref{thm:convergence-alt-scaling}. We start with its first part.
\begin{lemma}\label{lem:accuracy-alt-scaling}
The sequence
\begin{equation*}
    \VnE{r}{t}{i} \VnE{s}{t}{j} \norm{\At} \norm{\Bt} \qquad
    \text{for $t = 0, 1, \dotsc$}
\end{equation*}
is monotonically nonincreasing.
\end{lemma}
\begin{proof}
If step $t$ is an \Ostep{}, then
\begin{align*}
    \norm{\At} = \norm{\Bt} &= 1 \,, &
    \VnE{r}{t}{i} &= \VnE{r}{t-1}{i} \norm{\MnE{A}{t-1}{i,:}} \,, &
    \VnE{s}{t}{j} &= \VnE{s}{t-1}{j} \norm{\MnE{B}{t-1}{:,j}} \,,
\end{align*}
and therefore
\begin{align*}
    \VnE{r}{t}{i} \VnE{s}{t}{j} \norm{\At} \norm{\Bt}
        &= \VnE{r}{t-1}{i} \VnE{s}{t-1}{j}
            \norm{\MnE{A}{t-1}{i,:}} \norm{\MnE{B}{t-1}{:,j}} \\
        &\leq \VnE{r}{t-1}{i} \VnE{s}{t-1}{j}
            \norm{\At[t-1]} \norm{\Bt[t-1]} \,.
\end{align*}
Next, assume that step $t$ is an \Istep{}. Column $k$ of $\M{A}'$ is transformed so that
\begin{align*}
    \MnE{A}{t}{ik}
        &= \parenths*{\frac{
                \norm{\MnE{B}{t-1}{k,:}}
            }{
                \norm{\MnE{A}{t-1}{:,k}}
            }}^{\!\! \frac{1}{2}}
            \MnE{A}{t-1}{ik}
        \,, \\
    \intertext{and therefore}
    \norm{\MnE{A}{t}{:,k}}
        &= \parenths*{\frac{
                \norm{\MnE{B}{t-1}{k,:}}
            }{
                \norm{\MnE{A}{t-1}{:,k}}
            }}^{\!\! \frac{1}{2}}
            \! \norm{\MnE{A}{t-1}{:,k}}
        = \sqrt{
                \norm{\MnE{A}{t-1}{:,k}}
                \norm{\MnE{B}{t-1}{k,:}}
            }
        \numbereq\label{eqn:Istep-norm-col-At}
        \,, \\
    \intertext{and similarly}
    \norm{\MnE{B}{t}{k,:}}
        &= \sqrt{
                \norm{\MnE{A}{t-1}{:,k}}
                \norm{\MnE{B}{t-1}{k,:}}
            }
        \numbereq\label{eqn:Istep-norm-row-Bt}
    \,.
\end{align*}
Hence,
\begin{align*}
    \norm{\At} = \norm{\Bt}
        &= \max_k\sqrt{
                \norm{\MnE{A}{t-1}{:,k}}
                \norm{\MnE{B}{t-1}{k,:}}
            }
            \,, &
    \VnE{r}{t}{i} &= \VnE{r}{t-1}{i} \,, &
    \VnE{s}{t}{j} &= \VnE{s}{t-1}{j} \,,
\end{align*}
and therefore
\begin{align*}
    \VnE{r}{t}{i} \VnE{s}{t}{j} \norm{\At} \norm{\Bt}
        &= \VnE{r}{t-1}{i} \VnE{s}{t-1}{j}
            \parenths[\Big]{\max_k
                \sqrt{
                    \norm{\MnE{A}{t-1}{:,k}}
                    \norm{\MnE{B}{t-1}{k,:}}
                } \,
            }^2 \\
        &= \VnE{r}{t-1}{i} \VnE{s}{t-1}{j}
            \max_k\bigl(
                \norm{\MnE{A}{t-1}{:,k}}
                \norm{\MnE{B}{t-1}{k,:}}
            \bigr) \\
        &\leq \VnE{r}{t-1}{i} \VnE{s}{t-1}{j} \norm{\At[t-1]} \norm{\Bt[t-1]}
    \,.
\end{align*}
\end{proof}

Next, we prove that the factors in the sequence of \cref{thm:convergence-alt-scaling} converge individually. This is required in the subsequent analysis.
\begin{lemma}\label{lem:convergence-rti-stj-At-Bt}
The sequences $\VnE{r}{t}{i}$, $\VnE{s}{t}{j}$, $\norm{\At}$ and $\norm{\Bt}$ for $t = 0, 1, \dotsc$ converge.
\end{lemma}
\begin{proof}
As we show in the proof of \cref{lem:accuracy-alt-scaling},
\begin{align*}
    \norm{\At} = \norm{\Bt} &= 1 \,, \\
    \norm{\At[t+1]} = \norm{\Bt[t+1]}
        &= \max_k\sqrt{
                \norm{\MnE{A}{t}{:,k}}
                \norm{\MnE{B}{t}{k,:}}
            }
    && \text{for $t = t_0, t_0+2, \dotsc$} \,.
\end{align*}
Therefore
\begin{equation*}
    \norm{\At[t+1]}
        = \max_k\sqrt{
                \norm{\MnE{A}{t}{:,k}}
                \norm{\MnE{B}{t}{k,:}}
            }
        \leq \sqrt{\norm{\At}\norm{\Bt}} = 1
        \,,
\end{equation*}
and hence
\begin{equation*}
    \rpti{t+2}{i}
        = \norm{\MnE{A}{t+1}{i,:}} \leq \norm{\At[t+1]} \leq 1
         \,. \numbereq\label{eqn:rpti-leq-one}
\end{equation*}
The sequence $\VnE{r}{t}{i}$ satisfies
\begin{alignat*}{2}
    \VnE{r}{t_0}{i} &= \VnE{r}{t_0+1}{i} &&= \rpti{t_0}{i} \\
    \VnE{r}{t_0+2}{i} &= \VnE{r}{t_0+3}{i} &&= \rpti{t_0}{i} \rpti{t_0+2}{i} \\
                      &\vdotswithin{=}
                      &&\vdotswithin{=}
\end{alignat*}
It is nonnegative because $\rpti{t}{i} \geq 0$, it is monotonically nonincreasing by \cref{eqn:rpti-leq-one}, and hence it must converge. The same is true for $\VnE{s}{t}{j}$.

Consider the effect of the first $t$ steps of the iteration on $\M{A}'$ and $\M{B}'$. The cumulative effect of the \Osteps{} is to divide the rows of $\M{A}'$ by $\VnE{r}{t}{i}$ and the columns of $\M{B}'$ by $\VnE{s}{t}{j}$, and that of the \Isteps{} is to make sure that every column of $\M{A}'$ is equal in norm to the corresponding row of $\M{B}'$. Therefore
\begin{equation*}
    \MnE{A}{t}{ik}
        = \ME{A}{ik} \; \frac{1}{\VnE{r}{t}{i}}
            \parenths*{\frac{
                \max_j\abs[\big]{\ME{B}{kj} \big/ \VnE{s}{t-1}{j}}
            }{
                \max_i\abs[\big]{\ME{A}{ik} \big/ \VnE{r}{t-1}{i}}
        }}^{\!\! \frac{1}{2}}
    \qquad
    \text{for $t = t_0+1, t_0+2, \dotsc$}
    \,,
\end{equation*}
which shows that the convergence of $\VnE{r}{t}{i}$ and $\VnE{s}{t}{j}$ guarantees the convergence of $\MnE{A}{t}{ik}$, and hence also the convergence of $\norm{\At}$. The same is true for $\norm{\Bt}$.
\end{proof}

The following \lcnamecref{lem:convergence-wt} shows that the intermediate scaling factors that we compute in each step rapidly converge to 1. We use the notation
\begin{align*}
    \Sn{w}{t}
        &= \max\parenths[\Big]{
                \max_i\abs[\big]{\log\rpti{t}{i}},
                \max_j\abs[\big]{\log\sptj{t}{j}}
            }
        \,, \\
    \Sn{w}{t+1}
        &= \max_k\abs[\big]{\log\pptk{t+1}{k}}
        && \text{for $t = t_0, t_0+2, \dotsc$} \,.
\end{align*}
\begin{lemma}\label{lem:convergence-wt}
The following bounds hold:
\begin{align*}
    \Sn{w}{t} &\leq \Sn{w}{t-1} \,, \\
    \Sn{w}{t+1} &\leq 0.5 \Sn{w}{t}
        && \text{for $t = t_0+2, t_0+4, \dotsc$}
        \,.
\end{align*}
\end{lemma}
\begin{proof}
Assume that $t = t_0+2, t_0+4, \dotsc \,.$ Because step $t-2$ is an \Ostep{}, there is a column $g$ so that $\abs[\big]{\MnE{A}{t-2}{ig}} = 1$, and therefore
\begin{equation*}
    \rpti{t}{i}
        = \max_k\abs[\big]{\MnE{A}{t-1}{ik}}
        = \max_k\abs[\big]{\MnE{A}{t-2}{ik} \pptk{t-1}{k}}
        \geq \abs[\big]{\MnE{A}{t-2}{ig} \pptk{t-1}{g}}
        = \pptk{t-1}{g}
    \,.
\end{equation*}
Taking logarithms yields
\begin{align*}
    \log\rpti{t}{i} &\geq \log\pptk{t-1}{g} \,. \\
    \intertext{Both sides of this inequality are nonpositive because $\rpti{t}{i} \leq 1$ by \cref{eqn:rpti-leq-one}, and so}
    \abs[\big]{\log\rpti{t}{i}}
        = -\log\rpti{t}{i}
        &\leq -\log\pptk{t-1}{g}
        = \abs[\big]{\log\pptk{t-1}{g}}
    \,. \\
    \intertext{A similar analysis shows that}
    \abs[\big]{\log\sptj{t}{j}}
        &\leq \abs[\big]{\log\pptk{t-1}{f}}
    \,,
\end{align*}
for a suitably defined row $f$, and these two inequalities imply the first bound in the statement of the \lcnamecref{lem:convergence-wt}.

Next, let us prove the second bound. We have that
\begin{equation*}
    \parenths[\big]{\pptk{t+1}{k}}^2
        = \frac{
                \max_j\abs[\big]{\MnE{B}{t}{kj}}
            }{
                \max_i\abs[\big]{\MnE{A}{t}{ik}}
            }
        = \frac{
                \max_j\abs[\big]{\MnE{B}{t-1}{kj} \big/ \sptj{t}{j}}
            }{
                \max_i\abs[\big]{\MnE{A}{t-1}{ik} \big/ \rpti{t}{i}}
            }
        \leq \frac{
                \max_j\abs[\big]{\MnE{B}{t-1}{kj}} \;
                \max_j\bigl(1 \big/ \sptj{t}{j}\bigr)
            }{
                \max_i\abs[\big]{\MnE{A}{t-1}{ik} \big/ \rpti{t}{i}}
            }
        \,.
\end{equation*}
Inequality \cref{eqn:rpti-leq-one} states that $\rpti{t}{i} \leq 1$, and therefore
\begin{equation*}
    \max_i\abs[\big]{\MnE{A}{t-1}{ik} \big/ \rpti{t}{i}}
        \geq \max_i\abs[\big]{\MnE{A}{t-1}{ik}} \,,
\end{equation*}
which we substitute into the previous inequality, obtaining
\begin{equation*}
    \parenths[\big]{\pptk{t+1}{k}}^2
        \leq \frac{
                \max_j\abs[\big]{\MnE{B}{t-1}{kj}} \;
                \max_j\bigl(1 \big/ \sptj{t}{j}\bigr)
            }{
                \max_i\abs[\big]{\MnE{A}{t-1}{ik}}
            }
        \,.
\end{equation*}
By \cref{eqn:Istep-norm-col-At,eqn:Istep-norm-row-Bt}
\begin{equation*}
    \max_i\abs[\big]{\MnE{A}{t-1}{ik}}
        = \norm{\MnE{A}{t-1}{:,k}}
        = \norm{\MnE{B}{t-1}{k,:}}
        = \max_j\abs[\big]{\MnE{B}{t-1}{kj}} \,,
\end{equation*}
which implies
\begin{align*}
    \parenths[\big]{\pptk{t+1}{k}}^2
        &\leq \max\nolimits_j\parenths[\big]{1 \big/ \sptj{t}{j}}
        \,. \\
    \intertext{A similar analysis shows that}
    \parenths[\big]{\pptk{t+1}{k}}^2
        &\geq \frac{1}{\max_i\parenths[\big]{1 \big/ \rpti{t}{i}}}
    \,.
\end{align*}
Taking the logarithm of these two bounds and interchanging the positions of the logarithms with those of the $\max$ operators yields
\begin{equation*}
    -\max_i\parenths[\Big]{\log\parenths[\big]{1 \big/ \rpti{t}{i}}}
        \leq 2\log\pptk{t+1}{k}
        \leq \max_j\parenths[\Big]{\log\parenths[\big]{1 \big/ \sptj{t}{j}}}
    \,.
\end{equation*}
Because $\rpti{t}{i} \leq 1$ we have that $\log\parenths[\big]{1 \big/ \rpti{t}{i}} = \abs[\big]{\log\rpti{t}{i}}$, and similarly for $\sptj{t}{j}$. Applying this to the previous inequality yields
\begin{equation*}
    -\max_i\abs[\big]{\log\rpti{t}{i}}
        \leq 2\log\pptk{t+1}{k}
        \leq \max_j\abs[\big]{\log\sptj{t}{j}}
    \,,
\end{equation*}
and therefore
\begin{equation*}
    2\abs[\big]{\log\pptk{t+1}{k}}
        \leq \max\parenths[\Big]{
                \max_i\abs[\big]{\log\rpti{t}{i}},
                \max_j\abs[\big]{\log\sptj{t}{j}}
            }
    \,,
\end{equation*}
which implies the second bound in the statement of the \lcnamecref{lem:convergence-wt}.
\end{proof}

The following lemma proves linear convergence, and therefore completes the proof of \cref{thm:convergence-alt-scaling}.
\begin{lemma}\label{lem:convergence-nu}
There is a sequence $\Sn{\nu}{t}$ so that
\begin{align*}
    \VnE{\mu}{t}{ij} &\leq \Sn{\nu}{t} \,, &
        \VnE{\mu}{t+1}{ij} &\leq \Sn{\nu}{t+1} \,, \\
    \shortintertext{and}
    \Sn{\nu}{t+1} &= \Sn{\nu}{t} \,, &
    \Sn{\nu}{t+2} &\leq 0.5 \, \Sn{\nu}{t} &
        &\text{for $t = t_0, t_0+2, \dotsc .$}
\end{align*}
\end{lemma}
\begin{proof}
Assume that $t = t_0, t_0 + 2, \dotsc .$
Rearranging the definition of $\VnE{\mu}{t}{ij}$, we may write
\begin{equation*}
    \VnE{\mu}{t}{ij}
    = \parenths[\big]{
            \VnE{r}{t}{i} \VnE{s}{t}{j}
            \norm{\At} \norm{\Bt}
        }
        \,
        \parenths[\big]{
            \VnE{r}{\star}{i} \VnE{s}{\star}{j}
            \norm{\At[\star]} \norm{\Bt[\star]}
        }^{-1}
        -1 \,.
\end{equation*}
We have that
\begin{align*}
    \VnE{r}{t}{i}
        &= \rpti{t_0}{i} \rpti{t_0+2}{i} \dotsm \rpti{t}{i} \,, &
    \VnE{s}{t}{j}
        &= \sptj{t_0}{j} \sptj{t_0+2}{j} \dotsm \sptj{t}{j} \,, &
    \norm{\At} = \norm{\Bt} &= 1 \,, \\
    \intertext{and therefore}
    \VnE{r}{\star}{i}
        &= \rpti{t_0}{i} \rpti{t_0+2}{i} \dotsm \,, &
    \VnE{s}{\star}{j}
        &= \sptj{t_0}{j} \sptj{t_0+2}{j} \dotsm \,, &
    \norm{\At[\star]} = \norm{\Bt[\star]} &= 1 \,.
\end{align*}
Substituting this into the above yields
\begin{align*}
    \VnE{\mu}{t}{ij}
    &= \parenths[\big]{
            \rpti{t_0}{i} \sptj{t_0}{j}
            \rpti{t_0 + 2}{i} \sptj{t_0 + 2}{j}
            \dotsm
            \rpti{t}{i} \sptj{t}{j}
        }
        \,
        \parenths[\big]{
            \rpti{t_0}{i} \sptj{t_0}{j}
            \rpti{t_0 + 2}{i} \sptj{t_0 + 2}{j}
            \dotsm
        }^{-1}
        -1 \\
    &= \parenths[\big]{
            \rpti{t + 2}{i} \sptj{t + 2}{j}
            \rpti{t + 4}{i} \sptj{t + 4}{j}
            \dotsm
        }^{-1}
        -1 \,.
\end{align*}
Applying the definition of $\Sn{w}{t}$ to this, we obtain
\begin{align*}
    \VnE{\mu}{t}{ij} &= \parenths[\big]{
            \rpti{t + 2}{i} \sptj{t + 2}{j}
            \rpti{t + 4}{i} \sptj{t + 4}{j}
            \dotsm
        }^{-1}
        -1  \\
    &= \exp\parenths[\big]{
            -\log\rpti{t + 2}{i} - \log\sptj{t + 2}{j}
            - \log\rpti{t + 4}{i} - \log\sptj{t + 4}{j}
            - \dotsb
        }
        -1 \\
    &= \exp\parenths[\big]{
            \abs[\big]{\log\rpti{t + 2}{i}} + \abs[\big]{\log\sptj{t + 2}{j}}
            + \abs[\big]{\log\rpti{t + 4}{i}} + \abs[\big]{\log\sptj{t + 4}{j}}
            + \dotsb
        }
        -1 \\
    &\leq \exp\parenths[\big]{2\Sn{w}{t+2} + 2\Sn{w}{t+4} + \dotsb} - 1
    \,.
\end{align*}
We define
\begin{equation*}
    \Sn{\nu}{t}
        = \exp\parenths[\big]{2\Sn{w}{t+2} + 2\Sn{w}{t+4} + \dotsb} - 1
    \,, \qquad
    \Sn{\nu}{t+1} = \Sn{\nu}{t} \,,
\end{equation*}
thereby guaranteeing that $\VnE{\mu}{t}{ij} \leq \Sn{\nu}{t}$, as the \lcnamecref{lem:convergence-nu} states. Since $\VnE{r}{t}{i} \VnE{s}{t}{j} \norm{\At} \norm{\Bt}$ is monotonically nonincreasing, so is its relative distance to its limit, meaning that
\begin{align*}
    \VnE{\mu}{t+1}{ij} &\leq \VnE{\mu}{t}{ij} \\
    \shortintertext{and hence}
    \VnE{\mu}{t+1}{ij} &\leq \VnE{\mu}{t}{ij} \leq \Sn{\nu}{t} = \Sn{\nu}{t+1}
    \,.
\end{align*}
This proves another condition in the statement of the \lcnamecref{lem:convergence-nu}, leaving us only with the condition $\Sn{\nu}{t+2} \leq 0.5 \, \Sn{\nu}{t}$ to prove.

By \cref{lem:convergence-wt},
\begin{alignat*}{4}
    \Sn{w}{t+3} &\leq 0.5\Sn{w}{t+2} && \qquad &
        \Sn{w}{t+4} &\leq \Sn{w}{t+3} &&\leq 0.5\Sn{w}{t+2}
    \\
    \Sn{w}{t+5} &\leq 0.5\Sn{w}{t+4} &&\leq 0.25\Sn{w}{t+2} \qquad &
        \Sn{w}{t+6} &\leq \Sn{w}{t+5} &&\leq 0.25\Sn{w}{t+2} \\
    &\vdotswithin{\leq} &&\vdotswithin{\leq} &
        &\vdotswithin{\leq} &&\vdotswithin{\leq}
\end{alignat*}
Therefore
\begin{align*}
    \Sn{w}{t+2} &= \parenths{0.5 + 0.25 + \dotsb} \Sn{w}{t+2} \\
        &= 0.5\Sn{w}{t+2} + 0.25\Sn{w}{t+2} + \dotsb \\
        &\geq \Sn{w}{t+4} + \Sn{w}{t+6} + \dotsb \,,
    \intertext{and thus}
    \exp\parenths[\big]{2\Sn{w}{t+2}}
        &\geq \exp\parenths[\big]{2\Sn{w}{t+4} + 2\Sn{w}{t+6} + \dotsb} \,.
\end{align*}
Applying this bound to the definition of $\Sn{\nu}{t}$, we obtain
\begin{align*}
    \Sn{\nu}{t} &= \exp\parenths[\big]{2\Sn{w}{t+2} + 2\Sn{w}{t+4} + \dotsb} - 1 \\
    &= \exp\parenths[\big]{2\Sn{w}{t+2}}
        \exp\parenths[\big]{2\Sn{w}{t+4} + 2\Sn{w}{t+6} + \dotsb} - 1 \\
    &\geq \parenths[\Big]{
            \exp\parenths[\big]{2\Sn{w}{t+4} + 2\Sn{w}{t+6} + \dotsb}
        }^2 -1 \,.
\end{align*}
We define $x = \exp\parenths[\big]{2\Sn{w}{t+4} + 2\Sn{w}{t+6} + \dotsb}$, so that the above expression has the form $x^2 - 1$ and $\Sn{\nu}{t+2} = x - 1$. The rest of the proof is:
\begin{equation*}
    \Sn{\nu}{t} \geq x^2 - 1
    = \parenths{2 + x - 1}\parenths{x - 1}
    = \parenths[\big]{2 + \Sn{\nu}{t+2}} \, \Sn{\nu}{t+2}
    \geq 2 \, \Sn{\nu}{t+2} \,,
\end{equation*}
which implies that $\Sn{\nu}{t+2} \leq 0.5 \, \Sn{\nu}{t}$ as the \lcnamecref{lem:convergence-nu} states.
\end{proof}

Finally, we show that the analysis in \cref{lem:convergence-nu} is asymptotically sharp.
\begin{lemma}\label{lem:alt-scaling-example}
There are matrices $\M{A}$ and $\M{B}$ and indices $i$ and $j$ so that
\begin{equation*}
    \VnE{\mu}{t+1}{ij} = \VnE{\mu}{t}{ij} \,, \qquad
    \VnE{\mu}{t+2}{ij} \big/ \VnE{\mu}{t}{ij} \to 0.5 \qquad
    \text{for $t = t_0, t_0+2, \dotsc.$}
\end{equation*}
\end{lemma}
\begin{proof}
Let
\begin{equation*}
    \M{A} = \begin{bmatrix}
            1 & 0\\
            1 & 2^{-2^v}
        \end{bmatrix}
        \,, \qquad
    \M{B} = \begin{bmatrix}
            1 & 0\\
            0 & 1
        \end{bmatrix}
\end{equation*}
for some integer $v$, let us start the iteration with an \Ostep{}, and assume that $t = t_0, t_0+2, \dotsc.$ A straightforward calculation, which we omit for brevity, shows that
\begin{align*}
    \At &= \begin{bmatrix}
            1 & 0\\
            1 & 2^{-2^{v - (t - t_0)/2}}
        \end{bmatrix}
        \,, &
    \begin{bmatrix} \VnE{r}{t}{1} \\ \VnE{r}{t}{2} \end{bmatrix}
        &= \begin{bmatrix}
            \vphantom{\VnE{r}{t}{1}} 1 \\
            \vphantom{\VnE{r}{t}{2}} 1
        \end{bmatrix}
        \,, \\
    \Bt &= \begin{bmatrix}
            1 & 0\\
            0 & 1
        \end{bmatrix}
        \,, &
    \begin{bmatrix} \VnE{s}{t}{1} \\ \VnE{s}{t}{2} \end{bmatrix}
        &= \begin{bmatrix*}[l]
            \vphantom{\VnE{s}{t}{1}} 1 \\
            \vphantom{\VnE{s}{t}{2}} 2^{-2^v(1 - 2^{-(t-t_0)/2})}
        \end{bmatrix*}
        \,, \\
    \At[t+1] &= \begin{bmatrix}
            1 & 0\\
            1 & 2^{-2^{v - (t - t_0)/2 - 1}}
        \end{bmatrix}
        \,, &
    \begin{bmatrix} \VnE{r}{t+1}{1} \\ \VnE{r}{t+1}{2} \end{bmatrix}
        &= \begin{bmatrix} \VnE{r}{t}{1} \\ \VnE{r}{t}{2} \end{bmatrix}
        \,, \\
    \Bt[t+1] &= \begin{bmatrix}
            1 & 0\\
            0 & 2^{-2^{v - (t - t_0)/2 - 1}}
        \end{bmatrix}
        \,, &
    \begin{bmatrix} \VnE{s}{t+1}{1} \\ \VnE{s}{t+1}{2} \end{bmatrix}
        &= \begin{bmatrix} \VnE{s}{t}{1} \\ \VnE{s}{t}{2} \end{bmatrix}
        \,,
\end{align*}
and therefore
\begin{equation*}
    \At[\star] = \begin{bmatrix}
            1 & 0\\
            1 & 1
        \end{bmatrix}
        \,, \qquad
    \Bt[\star] = \begin{bmatrix}
            1 & 0\\
            0 & 1
        \end{bmatrix}
        \,, \qquad
    \begin{bmatrix} \VnE{r}{\star}{1} \\ \VnE{r}{\star}{2} \end{bmatrix}
        = \begin{bmatrix}
            \vphantom{\VnE{r}{\star}{1}} 1 \\
            \vphantom{\VnE{r}{\star}{2}} 1
        \end{bmatrix}
        \,, \qquad
    \begin{bmatrix} \VnE{s}{\star}{1} \\ \VnE{s}{\star}{2} \end{bmatrix}
        = \begin{bmatrix*}[l]
            \vphantom{\VnE{s}{\star}{1}} 1 \\
            \vphantom{\VnE{s}{\star}{2}} 2^{-2^v}
        \end{bmatrix*}
\end{equation*}
and
\begin{equation*}
    \norm{\At} = \norm{\Bt}
    = \norm{\At[t+1]} = \norm{\Bt[t+1]}
    = \norm{\At[\star]} = \norm{\Bt[\star]}
    = 1 \,.
\end{equation*}
Substituting the above into the definition of $\VnE{\mu}{t}{ij}$ yields
\begin{equation}\begin{aligned}\label{eqn:cvg-example-mu-22}
    \VnE{\mu}{t}{22} = \VnE{\mu}{t+1}{22} &= \frac{\abs[\big]{
            \VnE{r}{t}{2} \VnE{s}{t}{2} \norm{\At} \norm{\Bt}
            -
            \VnE{r}{\star}{2} \VnE{s}{\star}{2} \norm{\At[\star]} \norm{\Bt[\star]}
        }}{\abs[\big]{
            \VnE{r}{\star}{2} \VnE{s}{\star}{2} \norm{\At[\star]} \norm{\Bt[\star]}
        }} \\
    &= \frac{\abs[\big]{
            2^{-2^v(1 - 2^{-(t-t_0)/2})} - 2^{-2^v}
        }}{\abs[\big]{
            2^{-2^v}
        }} \\
    &= 2^{2^{v - (t - t_0)/2}} - 1\,.
\end{aligned}\end{equation}
Letting $x = 2^{2^{v - (t - t_0)/2 - 1}}$, so that the above expression has the form $x^2 - 1$ and $\VnE{\mu}{t+2}{22} = x - 1$, we have that
\begin{equation*}
    \VnE{\mu}{t}{22} = x^2 - 1 = (2 + x - 1)(x - 1)
    = \parenths[\big]{2 + \VnE{\mu}{t+2}{22}} \, \VnE{\mu}{t+2}{22} \,,
\end{equation*}
and therefore
\begin{equation*}
    \VnE{\mu}{t+2}{22} \big/ \VnE{\mu}{t}{22} = 1 \big/ \parenths[\big]{2 + \VnE{\mu}{t+2}{22}} \; \,.
\end{equation*}
From \cref{eqn:cvg-example-mu-22} we conclude that $\VnE{\mu}{t+2}{22} \to 0,$ and therefore $\VnE{\mu}{t+2}{22} \big/ \VnE{\mu}{t}{22} \to 0.5.$
\end{proof}

\section{Proof of \cref{thm:stopping-criterion}}
\label{app:stopping-criterion}

\begin{proof}
In the proof of \cref{lem:convergence-nu} (see \cref{app:convergence-alt-scaling}), we show that for all $i,j$ and $t=t_0,t_0+2, \dotsc$,
\begin{align*}
    \VnE{\mu}{t+1}{ij} &\leq \VnE{\mu}{t}{ij} \,, \\
    \VnE{\mu}{t}{ij}
        &\leq \exp\parenths[\big]{2\Sn{w}{t+2} + 2\Sn{w}{t+4} + \dotsb} - 1
        \,, \\
    \exp\parenths[\big]{2\Sn{w}{t+4} + 2\Sn{w}{t+6} + \dotsb}
        &\leq \exp\parenths[\big]{2\Sn{w}{t+2}}
        \,.
\end{align*}
Putting these three statements together yields
\begin{align*}
    \VnE{\mu}{t+1}{ij} &\leq \VnE{\mu}{t}{ij}
        \leq \exp\parenths[\big]{2\Sn{w}{t+2} + 2\Sn{w}{t+4} + \dotsb} - 1 \\
        &= \exp\parenths[\big]{2\Sn{w}{t+2}}
            \exp\parenths[\big]{2\Sn{w}{t+4} + 2\Sn{w}{t+6} + \dotsb} - 1 \\
        &\leq \exp\parenths[\big]{2\Sn{w}{t+2}}
            \exp\parenths[\big]{2\Sn{w}{t+2}} - 1 \\
        &= \exp\parenths[\big]{4\Sn{w}{t+2}} - 1 \,.
    \intertext{\Cref{lem:convergence-wt} guarantees that $\Sn{w}{t+2} \leq \Sn{w}{t+1}$, and substituting this into the above yields}
    \VnE{\mu}{t+1}{ij} &\leq \exp\parenths[\big]{4\Sn{w}{t+1}} - 1
        \,. \numbereq\label{eqn:mu-bound-Istep}
    \intertext{Similarly,}
    \VnE{\mu}{t+2}{ij}
        &\leq \exp\parenths[\big]{2\Sn{w}{t+4} + 2\Sn{w}{t+6} + \dotsb} - 1
        \leq \exp\parenths[\big]{2\Sn{w}{t+2}} - 1
        \,. \numbereq\label{eqn:mu-bound-Ostep}
\end{align*}

Next, let us prove the first statement of the \lcnamecref{thm:stopping-criterion}. Assume that
\begin{equation*}
    (1 + \tau)^{-\frac{1}{4}} \leq \pptk{t+1}{k}
        \leq (1 + \tau)^\frac{1}{4}
\end{equation*}
for all $k$.
Taking logarithms yields
\begin{equation*}
    -0.25\log(1 + \tau) \leq \log\pptk{t+1}{k}
        \leq 0.25\log(1 + \tau) \,,
\end{equation*}
or equivalently
\begin{equation*}
    \abs[\big]{\log\pptk{t+1}{k}} \leq 0.25\log(1 + \tau) \,,
\end{equation*}
and therefore
\begin{equation*}
    \Sn{w}{t+1} = \max_k\abs[\big]{\log\pptk{t+1}{k}} \leq 0.25\log(1 + \tau) \,.
\end{equation*}
Substituting this into \cref{eqn:mu-bound-Istep}, we find that
\begin{equation*}
    \VnE{\mu}{t+1}{ij} \leq \exp\parenths[\big]{4\Sn{w}{t+1}} - 1
        \leq \exp\parenths[\big]{4 \cdot 0.25 \log(1 + \tau)} - 1
        = \tau \,,
\end{equation*}
for all $i,j$, which proves the first statement of the \lcnamecref{thm:stopping-criterion}.

Let us prove the second statement of the \lcnamecref{thm:stopping-criterion}. Assume that
\begin{equation*}
    (1 + \tau)^{-\frac{1}{2}} \leq \rpti{t+2}{i} \,, \qquad
        (1 + \tau)^{-\frac{1}{2}} \leq \sptj{t+2}{j} \\
\end{equation*}
for all $i$ and $j$.
Taking logarithms yields
\begin{equation*}
    -0.5\log(1 + \tau) \leq \log\rpti{t+2}{i} \,, \qquad
        -0.5\log(1 + \tau) \leq \log\sptj{t+2}{j} \,.
\end{equation*}
We show in the proof of \cref{lem:convergence-rti-stj-At-Bt} that  $\rpti{t+2}{i} \leq 1$, and therefore $\log\rpti{t+2}{i} \leq 0$ and similarly for $\sptj{t+2}{j}$. Hence
\begin{equation*}
    \abs[\big]{\log\rpti{t+2}{i}} \leq 0.5\log(1 + \tau) \,, \qquad
        \abs[\big]{\log\sptj{t+2}{j}} \leq 0.5\log(1 + \tau) \,,
\end{equation*}
and hence
\begin{equation*}
    \Sn{w}{t+2} = \max\parenths[\Big]{
        \max_i \abs[\big]{\log\rpti{t+2}{i}},
        \max_j \abs[\big]{\log\sptj{t+2}{j}}
        }
    \leq 0.5\log(1 + \tau) \,.
\end{equation*}
Substituting this into \cref{eqn:mu-bound-Ostep} yields
\begin{equation*}
    \VnE{\mu}{t+2}{ij} \leq \exp\parenths[\big]{2\Sn{w}{t+2}} - 1
    \leq \exp\parenths[\big]{2 \cdot 0.5\log(1+\tau)} - 1
    = \tau \,,
\end{equation*}
for all $i,j$, which proves the second statement of the \lcnamecref{thm:stopping-criterion}.
\end{proof}

\end{document}